\documentclass[11pt]{amsart} 

\usepackage{amssymb,cite,hyperref}
\usepackage[margin=1in]{geometry}
\usepackage{tensor,enumitem,verbatim}
\usepackage[all,cmtip]{xy}
\usepackage{todonotes}

\newcommand{\cbinom}[2]{\genfrac{\{}{\}}{0pt}{}{#1}{#2}}

\newcommand{\subgrp}[1]{\langle #1 \rangle}
\newcommand{\set}[1]{\left\{ #1 \right\}}
\newcommand{\abs}[1]{\left| #1 \right|}
\newcommand{\bs}[1]{\boldsymbol{#1}}
\newcommand{\wh}[1]{\widehat{ #1}}
\newcommand{\wt}[1]{\widetilde{ #1}}
\newcommand{\ul}[1]{\underline{#1}}
\newcommand{\ol}[1]{\overline{#1}}

\newcommand{\ev}{\textup{ev}}
\newcommand{\odd}{\textup{odd}}
\newcommand{\ve}{\varepsilon}
\newcommand{\ualpha}{\ul{\alpha}}

\newcommand{\sslash}{/\!\!/}

\DeclareMathOperator{\Aut}{Aut}

\DeclareMathOperator{\Ext}{Ext}

\DeclareMathOperator{\opH}{H}
\DeclareMathOperator{\Hom}{Hom}
\DeclareMathOperator{\id}{id}
\DeclareMathOperator{\im}{im}

\DeclareMathOperator{\Mat}{Mat}
\DeclareMathOperator{\Max}{Max}

\DeclareMathOperator{\soc}{soc}

\newcommand{\bfAut}{\mathbf{Aut}}
\newcommand{\bfHom}{\mathbf{Hom}}

\newcommand{\gotimes}{\tensor[^g]{\otimes}{}}

\newcommand{\Hbul}{\opH^\bullet}

\newcommand{\G}{\mathbb{G}}
\newcommand{\M}{\mathbb{M}}
\newcommand{\N}{\mathbb{N}}
\renewcommand{\P}{\mathbb{P}}
\newcommand{\Z}{\mathbb{Z}}

\newcommand{\Ga}{\G_a}
\newcommand{\Gam}{\Ga^-}
\newcommand{\Gar}{\G_{a(r)}}
\newcommand{\Gaone}{\G_{a(1)}}

\newcommand{\calI}{\mathcal{I}}
\newcommand{\calN}{\mathcal{N}}

\newcommand{\fm}{\mathfrak{m}}

\newcommand{\zero}{\ol{0}}
\newcommand{\one}{\ol{1}}
\newcommand{\Azero}{A_{\zero}}

\newcommand{\Vone}{V_{\one}}
\newcommand{\Vzero}{V_{\zero}}

\newcommand{\alg}{\mathfrak{alg}}
\newcommand{\calg}{\mathfrak{calg}}
\newcommand{\csalg}{\mathfrak{csalg}}
\newcommand{\salg}{\mathfrak{salg}}

\newcommand{\fsmod}{\mathfrak{smod}}

\newcommand{\Matmn}{\Mat_{m|n}}

\newcommand{\Mones}{\M_{1;s}}
\newcommand{\Mr}{\M_r}
\newcommand{\Mone}{\M_1}

\newcommand{\Moneone}{\M_{1;1}}
\newcommand{\Mrone}{\M_{r;1}}
\newcommand{\Mrf}{\M_{r;f}}
\newcommand{\Mrfeta}{\M_{r;f,\eta}}
\newcommand{\Mrseta}{\M_{r;s,\eta}}
\newcommand{\Mrs}{\M_{r;s}}
\newcommand{\Mrt}{\M_{r;t}}
\newcommand{\Nr}{\calN_r}

\newcommand{\NrG}{\calN_r(G)}
\newcommand{\None}{\calN_1}
\newcommand{\NoneG}{\None(G)}
\newcommand{\Nrs}{N_{r;s}}
\newcommand{\Nrone}{N_{r;1}}
\renewcommand{\Pr}{\P_r}
\newcommand{\Pone}{\P_1}

\newcommand{\bsN}{\bs{\mathcal{N}}}
\newcommand{\bsNr}{\bsN_r}
\newcommand{\bsVr}{\bsV_r}
\newcommand{\bsVrs}{\bsV_{r;s}}

\newcommand{\Amn}{A^{m|n}}

\newcommand{\GLmn}{GL_{m|n}}
\newcommand{\GLmnr}{GL_{m|n(r)}}

\newcommand{\bsa}{\bs{a}}
\newcommand{\bsmu}{\bs{\mu}}
\newcommand{\bsb}{\bs{b}}
\newcommand{\bsF}{\bs{F}}
\newcommand{\bsV}{\bs{V}}

\numberwithin{equation}{subsection}

\newtheorem{theorem}{Theorem}[subsection]
\newtheorem{proposition}[theorem]{Proposition}
\newtheorem{corollary}[theorem]{Corollary}
\newtheorem{lemma}[theorem]{Lemma}
\newtheorem{question}[theorem]{Question}
\newtheorem{conjecture}[theorem]{Conjecture}

\theoremstyle{definition}
\newtheorem{definition}[theorem]{Definition}

\newtheorem{remark}[theorem]{Remark}

\title[Spectrum and support for infinitesimal unipotent supergroup schemes]{On the cohomological spectrum and support varieties for infinitesimal unipotent supergroup schemes}

\author{Christopher M.\ Drupieski}
\address{Department of Mathematical Sciences,
		DePaul University,
		Chicago, IL 60614, USA}
\email{c.drupieski@depaul.edu}

\author{Jonathan R. Kujawa}
\address{Department of Mathematics \\
		University of Oklahoma \\
		Norman, OK 73019, USA}
\email{kujawa@math.ou.edu}

\thanks{The first author was supported in part by a Simons Collaboration Grant for Mathematicians, and by NSF Grant No.\ DMS-1440140 while he was in residence at the Mathematical Sciences Research Institute in Berkeley, CA, during the Spring 2018 semester. The second author was supported in part by NSA grant
H98230-16-0055 and in part by a Simons Collaboration Grant for Mathematicians.}


\allowdisplaybreaks

\begin{document}

\begin{abstract}
We show that if $G$ is an infinitesimal elementary supergroup scheme of height $\leq r$, then the cohomological spectrum $\abs{G}$ of $G$ is naturally homeomorphic to the variety $\NrG$ of supergroup homo\-morphisms $\rho: \Mr \rightarrow G$ from a certain (non-algebraic) affine supergroup scheme $\Mr$ into $G$. In the case $r=1$, we further identify the cohomological support variety of a finite-dimensional $G$-supermodule $M$ as a subset of $\NoneG$. We then discuss how our methods, when combined with recently-announced results by Benson, Iyengar, Krause, and Pevtsova, can be applied to extend the homeomorphism $\NrG \cong \abs{G}$ to arbitrary infinitesimal unipotent supergroup schemes.
\end{abstract}

\maketitle

\section{Introduction}

\subsection{Overview}

Let $k$ be a field of positive characteristic $p$. For more than thirty years, support varieties have played a prominent role in relating the $p$-modular representation theory of a finite group $G$ (and related algebraic structures) to the ambient geometry encoded by the spectrum $\abs{G}$ of the cohomology ring $\Hbul(G,k)$. Unipotent subgroups have frequently been important intermediate actors in this relationship. For example, Quillen \cite{Quillen:1971} showed that the cohomological spectrum of a finite group is stratified by pieces coming from its elementary abelian $p$-subgroups. Later, building on work by Friedlander and Parshall \cite{Friedlander:1987,Friedlander:1986b} for finite-dimensional restricted Lie algebras, Suslin, Friedlander, and Bendel (SFB) \cite{Suslin:1997,Suslin:1997a} showed that if $G$ is an infinitesimal $k$-group scheme of height $\leq r$, then $\abs{G}$ is homeomorphic to the variety of infinitesimal one parameter subgroups $\nu: \Gar \rightarrow G$, i.e., the variety of homomorphisms into $G$ from the $r$-th Frobenius kernel of the additive group scheme $\Ga$. More generally, Friedlander and Pevtsova \cite{Friedlander:2007} related the cohomological spectrum of an arbitrary finite $k$-group scheme $G$ to the space $\Pi(G)$ of $\pi$-points of $G$, which consists of equivalence classes of flat $K$-algebra maps $\alpha_K: K[t]/(t^p) \rightarrow KG$ (for $K$ a field extension of $k$) that factor through the group algebra of an abelian unipotent subgroup scheme of $G_K$.

Now suppose $p \geq 3$. This paper is a continuation of our work in \cite{Drupieski:2017a} investigating the cohomology of finite $k$-supergroup schemes. Recall that an affine $k$-supergroup scheme is a representable functor from the category $\csalg_k$ of commutative $k$-superalgebras to groups. An affine $k$-supergroup scheme $G$ is finite if its representing (Hopf) superalgebra $k[G]$ is finite-dimensional, and is infinitesimal if it is finite and if the augmentation ideal of $k[G]$ is nilpotent. In \cite{Drupieski:2017a} we introduced a family $\Mrfeta$ of infinitesimal supergroup schemes, which we called multiparameter supergroups, that are parametrized by an integer $r \geq 1$, a $p$-polynomial $0 \neq f \in k[T]$ without constant term, and a scalar $\eta \in k$. The definitions of these multiparameter supergroups are recalled in Section \ref{subsection:conventions}. If $f = T^{p^s}$ is a single monomial (and if $r \geq 2$ when $\eta \neq 0$), then $\Mrseta := \M_{r;T^{p^s},\eta}$ is unipotent, but in general the group algebra $k\Mrfeta$ has a nontrivial semisimple subalgebra. Following Suslin, Friedlander, and Bendel \cite{Suslin:1997}, in \cite{Drupieski:2017a} we defined characteristic extension classes for the general linear supergroup $\GLmn$, and calculated how the classes restricted along homomorphisms $\rho: \Mrfeta \rightarrow \GLmn$. We then applied our calculations to describe (for $k$ algebraically closed) the maximal ideal spectrum $\abs{\GLmnr}$ of the cohomology ring for the $r$-th Frobenius kernel of $\GLmn$. Roughly speaking (and modulo a finite morphism), we showed that $\abs{\GLmnr}$ is equal to
	\[
	\bigcup_{f,\eta} \bfHom(\Mrfeta,\GLmnr)(k),
	\]
where $\bfHom(\Mrfeta,\GLmnr)(k)$ denotes the variety of homomorphisms $\phi: \Mrfeta \rightarrow \GLmnr$.

In the present work we focus our attention on the \emph{unipotent} multiparameter supergroups, or what is more-or-less the same, the infinitesimal objects in the family of \emph{elementary} finite super\-group schemes, as recently classified (for $k$ perfect) by Benson, Iyengar, Krause, and Pevtsova (BIKP). The only elementary supergroup schemes that are not enumerated among our multiparameter supergroups are $\Gar$ for $r \geq 1$, and the odd additive group scheme $\Gam$, though one has $\M_{r;1} := \M_{r;T^p,0} = \Gar \times \Gam$. As we point out in Lemma \ref{lemma:Mrquotient}, the (height-$r$) infinitesimal elementary supergroup schemes are all quotients of a common supergroup scheme, $\Mr$, which we initially defined in \cite[\S3.1]{Drupieski:2017a}. The supergroup $\Mr$ then plays a central role in the rest of the paper.

One of our main results, presented in Corollary \ref{corollary:psirelementary}, is that if $k$ is algebraically closed, and if $G$ is a height-$r$ infinitesimal elementary supergroup scheme, then the cohomological spectrum $\abs{G}$ of $G$ is naturally homeomorphic to the variety
	\[
	\NrG := \bfHom(\Mr,G)(k) := \Hom_{Grp/k}(\Mr,G)
	\]
of all $k$-supergroup scheme homomorphisms $\rho: \Mr \rightarrow G$. For $r=1$ (and $G$ still infinitesimal elementary) we further identify, in Theorem \ref{theorem:supportheightone}, the support variety $\abs{G}_M$ of a finite-dimensional $G$-supermodule $M$ with the set
	\begin{align*}
	\NoneG_M :&= \set{ \phi \in \NoneG : \id_{\Mone}(\phi^* M) = \infty} \\
	&= \set{ \phi \in \NoneG : \Ext_{\Mone}^i(M,M) \neq 0 \text{ for infinitely many $i \geq 1$}}.
	\end{align*}
Here $\phi^* M$ denotes the pullback of $M$ along the homomorphism $\phi: \Mone \rightarrow G$, and $\id_{\Mone}(\phi^* M)$ denotes the injective dimension of $\phi^* M$ in the category of rational $\Mone$-supermodules. Our definition of the support set $\NoneG_M$ is inspired by similar definitions appearing in the literature in the context of commutative local rings (cf.\ \cite{Avramov:1989,Avramov:2000,Jorgensen:2002}), and which were brought to our attention by way of a talk by Srikanth Iyengar at the Conference on Groups, Representations, and Cohomology, held at Sabal M\`{o}r Ostaig, Isle of Skye, Scotland, in June 2015. At present we do not understand how the existing support theory for local commutative rings relates to the rational cohomology of $\Mone$, but an obvious intermediate actor is the (infinite-dimensional) group algebra $k\Mone := k[\Mone]^\#$ and its `polynomial subalgebra' $\Pone$. We expect that a better understanding of the relationship will help answer, for example, the question of whether or not the support sets of Section \ref{subsection:supportset} are Zariski closed subsets of $\NrG$.

At the end of the paper we discuss how our methods can be extended to arbitrary infinitesimal unipotent supergroup schemes when they are combined with recently-announced results by Benson, Iyengar, Krause, and Pevtsova concerning the detection of nilpotent elements in cohomology. The first anticipated application, which we state as Conjecture \ref{conjecture:spectrum}, is that for an arbitrary infinitesimal unipotent supergroup scheme $G$ of height $\leq r$, there exists a natural homeomorphism
	\[
	\abs{G} \cong \NrG = \Hom_{Grp/k}(\Mr,G).
	\]
This suggests that, at least for infinitesimal unipotent supergroups, $\NrG$ is the correct analogue of SFB's variety of infinitesimal one-parameter subgroups. We state this result as a conjecture rather than as a theorem only because, at the time of writing this article, the BIKP detection theorem has only been announced as a preliminary result, and has not yet appeared in preprint form; otherwise we provide all of the details to justify the conjecture. The second anticipated application, stated as Conjecture \ref{conjecture:supportheightone}, is an extension of Theorem \ref{theorem:supportheightone} to arbitrary height-$1$ infinitesimal unipotent supergroup schemes. To justify Conjecture \ref{conjecture:supportheightone} we must assume that the support set $\NoneG_M$ is a Zariski closed subset of $\NoneG$, and we must also assume a technical condition concerning the ideal of functions defining $\NoneG_M$ as a subset of $\NoneG$.

We do not make any particular speculations concerning support varieties for infinitesimal unipotent supergroups of height greater than one, though we expect that it should be possible to use methods like those employed by Suslin, Friedlander, and Bendel \cite{Suslin:1997a} to bootstrap up from height-$1$ unipotent supergroups to arbitrary infinitesimal unipotent supergroups. For example, for $r \geq 2$ and $0 \neq \eta \in k$, there are superalgebra (though not Hopf superalgebra) isomorphisms
	\[
	k\Mrseta \cong k\M_{r-1;s+1} \quad \text{and} \quad	k\Mrs \cong (k\Gaone)^{\otimes (r-1)} \otimes k\Mones,
	\]
which could enable one, as in \cite[Proposition 6.5]{Suslin:1997a}, to reduce support variety calculations for arbitrary infinitesimal elementary supergroup schemes to calculations for $\Gaone$ and $\Mones$. More generally, if $\Mrfeta$ is an arbitrary multiparameter supergroup and if $A = k\Mrseta$, then there exists a semisimple Hopf subalgebra $B \subset A$ such that the Hopf quotient $A\sslash B$ is isomorphic as a superalgebra to $k\M_{r';s'}$ for some integers $r',s' \geq 0$. Thus, one could expect to reduce support variety calculations for arbitrary multiparameter supergroups to the unipotent case.

\subsection{Organization of the paper}

Section \ref{section:homomorphisms} is devoted to foundational results concerning the set $\NrG$ of supergroup scheme homomorphisms $\rho: \Mr \rightarrow G$. In Section \ref{subsection:Homsets} we significantly extend the calculations of \cite[Lemma 3.3.2]{Drupieski:2017a} by computing for each purely even reduced commutative $k$-algebra $A$ and for each infinitesimal elementary supergroup scheme $G$ the set
	\[
	\bfHom(\Mr,G)(A) = \Hom_{Grp/A}(\Mr \otimes_k A,G \otimes_k A)
	\]
of all $A$-supergroup scheme homomorphisms $\rho: \Mr \otimes_k A \rightarrow G \otimes_k A$. In Section \ref{subsection:nilpotentsupermatrices} we show that if the field $k$ is algebraically closed, and if $G$ is an algebraic $k$-supergroup scheme, then the set $\NrG$ admits the structure of an affine algebraic variety. In particular, $\Nr(\GLmn)$ may be identified with a certain variety of commuting nilpotent supermatrices. Then following SFB, in Section \ref{subsection:universalhomomorphism} we define for each algebraic $k$-supergroup scheme $G$ the universal supergroup homomorphism from $\Mr$ to $G$, and use it to define a homomorphism of graded $k$-algebras $\psi_r: H(G,k) \rightarrow k[\NrG]$.

In Section \ref{section:cohomologyandsupport} we investigate the map of varieties $\Psi: \NrG \rightarrow \abs{G}$ induced by the homomorphism $\psi_r$. In Section \ref{subsection:inducedmaps} we show that $\Psi$ is a homeomorphism if $G$ is a height-$r$ infinitesimal elementary supergroup scheme. Then after making some observations in Section \ref{subsection:LHS} arising from various Lyndon--Hochschild--Serre spectral sequences, in Section \ref{subsection:supportset} we define the support sets $\NrG_M$, and in Section \ref{subsection:supportheightone} we show that $\NoneG_M$ identifies with the cohomological support variety $\abs{G}_M$ when $G$ is a height-one infinitesimal elementary supergroup scheme. Finally, in Section \ref{subsection:BIKPapplications} we discuss applications of the detection theorem recently announced by BIKP.

\subsection{Conventions} \label{subsection:conventions}

We generally follow the conventions of our previous work \cite{Drupieski:2016a,Drupieski:2017a}, to which we refer the reader for any unexplained terminology or notation. For additional standard terminology, notation, and conventions, the reader may consult Jantzen's book \cite{Jantzen:2003}.

Except when indicated otherwise, $k$ will denote a field of characteristic $p \geq 3$, $r$ will denote a positive integer, and $\N = \set{0,1,2,3,\ldots}$ will denote the set of nonnegative integers. All vector spaces will be $k$-vector spaces, and all unadorned tensor products will denote tensor products over $k$. Given a $k$-vector space $V$, let $V^\#$ be its $k$-linear dual $\Hom_k(V,k)$. Set $\Z_2 = \Z/2\Z = \set{\zero,\one}$. Following the literature, we use the prefix `super' to indicate that an object is $\Z_2$-graded. We denote the decomposition of a vector superspace (superalgebra, etc.) $V$ into its $\Z_2$-graded homogeneous components by $V = \Vzero \oplus \Vone$, calling $\Vzero$ the even subspace of $V$ and $\Vone$ the odd subspace of $V$, and write $\ol{v} \in \Z_2$ to denote the $\Z_2$-degree of a homogeneous element $v \in V$. We use the symbol $\cong$ to denote even (i.e., degree-preserving) isomorphisms of superspaces, and reserve the symbol $\simeq$ for odd (i.e., degree-reversing) isomorphisms of superspaces.

For the reader's convenience we recall here the definitions of the multiparameter supergroups introduced in \cite{Drupieski:2017a}. First, $\Mr$ is the affine $k$-supergroup scheme whose coordinate superalgebra $k[\Mr]$ is the commutative $k$-super\-algebra generated by the odd element $\tau$ and the even elements $\theta$ and $\sigma_i$ for $i \in \N$, such that $\tau^2 = 0$, $\sigma_0 = 1$, $\theta^{p^{r-1}} = \sigma_1$, and $\sigma_i \sigma_j = \binom{i+j}{i} \sigma_{i+j}$ for $i,j \in \N$:
\[ \textstyle
k[\Mr] = k[\tau,\theta,\sigma_1,\sigma_2,\ldots]/\subgrp{ \tau^2=0, \theta^{p^{r-1}} = \sigma_1, \text{ and } \sigma_i\sigma_j = \binom{i+j}{i}\sigma_{i+j} \text{ for $i,j \in \N$}}.
\]
The coproduct $\Delta$ and the antipode $S$ on $k[\Mr]$ are given by the formulas
	\begin{align*}
	\Delta(\tau) &= \tau \otimes 1 + 1 \otimes \tau, & S(\tau) &= -\tau, \\
	\Delta(\theta) &= \theta \otimes 1 + 1 \otimes \theta, & S(\theta) &= -\theta, \\
	\Delta(\sigma_i) &= \textstyle \sum_{u+v=i} \sigma_u \otimes \sigma_v + \sum_{u+v+p=i} \sigma_u \tau \otimes \sigma_v \tau, & S(\sigma_i) &= (-1)^i \sigma_i.
	\end{align*}
The coproduct on $k[\Mr]$ induces a superalgebra structure on its $k$-linear dual $k\Mr := k[\Mr]^\#$. With this structure, we call $k\Mr$ the \emph{group algebra} of $\Mr$. By \cite[Proposition 3.1.4]{Drupieski:2017a}, the group algebra $k\Mr$ is given by
	\begin{equation} \label{eq:kMr}
	k\Mr = k[[u_0,\ldots,u_{r-1},v]]/\subgrp{u_0^p,\ldots,u_{r-2}^p,u_{r-1}^p + v^2},
	\end{equation}
where the superdegrees of elements are given by $\ol{u_i} = \zero$ for each $i$ and $\ol{v} = \one$.

Let $\Pr$ be the `polynomial subalgebra' of $k\Mr$,
	\begin{equation} \label{eq:kMrpol}
	\Pr = k[u_0,\ldots,u_{r-1},v]/\subgrp{u_0^p,\ldots,u_{r-2}^p,u_{r-1}^p + v^2}.
	\end{equation}
By \cite[Remark 3.1.3(3)]{Drupieski:2017a}, the $\Z_2$-grading on $k[\Mr]$ lifts to a $\Z$-grading, which makes $k[\Mr]$ into a graded Hopf algebra of finite type in the sense of Milnor and Moore \cite{Milnor:1965}. Then $\Pr$ is the graded dual of $k[\Mr]$, and hence $\Pr$ inherits by duality a Hopf superalgebra structure \cite[Proposition 4.8]{Milnor:1965}. The elements $u_0$, $u_{r-1}^p$, and $v$ are each primitive with respect to the coproduct on $\Pr$. Then given a nonzero $p$-polynomial $f \in k[T]$ without constant term, and given a scalar $\eta \in k$, the element $f(u_{r-1})+\eta \cdot u_0$ is primitive in $\Pr$, and hence the quotient
	\[
	k\Mrfeta := \Pr/\subgrp{f(u_{r-1})+\eta \cdot u_0}
	\]
is a finite-dimensional (super)cocommutative Hopf superalgebra. Now the multiparameter supergroup $\Mrfeta$ is the affine $k$-supergroup scheme such that $k[\Mrfeta]^\# = k\Mrfeta$. Set $\Mrf = \M_{r;f,0}$, and given an integer $s \geq 1$ set $\Mrseta = \M_{r;T^{p^s},\eta}$ and $\Mrs = \M_{r;T^{p^s},0}$. Then $k[\Mrs]$ is the Hopf subsuperalgebra of $k[\Mr]$ generated by $\tau$, $\theta$, and $\sigma_i$ for $1 \leq i < p^s$.

\subsection{Acknowledgements} \label{SS:Acknowledgements}

The authors are pleased to thank David Benson, Srikanth Iyengar, and Julia Pevtsova for enlightening conversations that helped contribute to this work. In particular, the authors thank Julia Pevtsova for explaining the BIKP detection theorem (Theorem \ref{theorem:BIKP}), and they thank Srikanth Iyengar for generously sharing a preprint copy of his work with Luchezar Avramov on restricting homology to hypersurfaces. The first author thanks the organizers of the Southern Regional Algebra Conference, held March 17--19, 2017, at the University of South Alabama, for giving him the opportunity to speak on the work in \cite{Drupieski:2017a}.

\section{Varieties of homomorphisms} \label{section:homomorphisms}

\subsection{Elementary supergroup schemes}

The following definition (stated using slightly different though equivalent terminology) is due to Benson, Iyengar, Krause, and Pevtsova (BIKP).

\begin{definition}[Elementary supergroup schemes]
A finite supergroup scheme is \emph{elementary} if it is isomorphic to a quotient of $\Mrs \times (\Z/p\Z)^{\times t}$ for some integers $r,s,t \geq 1$.
\end{definition}

In the preceding definition, $\Z/p\Z$ denotes the finite supergroup scheme whose group algebra is $k(\Z/p\Z)$, the ordinary (purely even) group algebra over $k$ of the finite cyclic group $\Z/p\Z$. The next theorem was announced (using slightly different terminology) by Julia Pevtsova at the Mathematical Congress of the Americas in Montreal in July 2017, in the talk she gave in the Special Session on Cohomology of Groups. By convention, if $r=0$ then $\Gar$ is the trivial group.

\begin{theorem}[BIKP] \label{theorem:elementaryclassification}
Suppose $k$ is perfect. Then each elementary finite $k$-supergroup scheme is isomorphic to one of the following (unipotent) $k$-supergroup schemes:
	\begin{itemize}
	\item $\Gar \times (\Z/p\Z)^{\times t}$ for some integers $r,t \geq 0$;
	\item $\Gar \times \Gam \times (\Z/p\Z)^{\times t}$ for some integers $r,t \geq 0$;
	\item $\Mrs \times (\Z/p\Z)^{\times t}$ for some integers $r,s \geq 1$ and $t \geq 0$; or
	\item $\Mrseta \times (\Z/p\Z)^{\times t}$ for some integers $r \geq 2$, $s \geq 1$, $t \geq 0$, and some $0 \neq \eta \in k$.
	\end{itemize}
\end{theorem}

\begin{corollary} \label{corollary:infinitesimalelementary}
Suppose $k$ is perfect. Then each infinitesimal elementary $k$-supergroup scheme is isomorphic to one of the following infinitesimal (unipotent) $k$-supergroup schemes:
	\begin{itemize}
	\item $\Gar $ for some integer $r \geq 0$;
	\item $\Gar \times \Gam$ for some integer $r \geq 0$;
	\item $\Mrs$ for some integers $r,s \geq 1$; or
	\item $\Mrseta$ for some integers $r \geq 2$, $s \geq 1$, and some $0 \neq \eta \in k$.
	\end{itemize}	
\end{corollary}

Recall that a $k$-supergroup scheme $G$ is \emph{algebraic} if $k[G]$ is finitely-generated as a $k$-algebra.

\begin{lemma} \label{lemma:Mrquotient}
Suppose $k$ is perfect, and let $G$ be an algebraic $k$-supergroup scheme. Then $G$ is a quotient of $\Mr$ for some $r \geq 1$ if and only if $G$ is an infinitesimal elementary supergroup scheme.
\end{lemma}

\begin{proof}
First suppose $\phi: \Mr \rightarrow G$ is a quotient homomorphism. By \cite[Remark 3.1.3(4)]{Drupieski:2017a}, $\phi$ factors through the canonical quotient map $\pi_{r;s}: \Mr \twoheadrightarrow \Mrs$ for some $s \geq 1$. Then $G$ is a quotient of $\Mrs$ for some $s \geq 1$, and hence $G$ is an elementary supergroup scheme. Conversely, each supergroup listed in Corollary \ref{corollary:infinitesimalelementary} is a quotient of $\Mr$ for some $r \geq 1$; cf.\ Remark \ref{rem:Hom}\eqref{item:quotientmaps} below.
\end{proof}

\subsection{Homomorphisms between infinitesimal elementary supergroups} \label{subsection:Homsets}

Recall that if $G$ is an affine $k$-supergroup scheme and if $A \in \csalg_k$ is a commutative $k$-superalgebra, then $G \otimes_k A$ denotes the affine $A$-supergroup scheme with coordinate superalgebra $A[G] := k[G] \otimes_k A$. Given affine super\-group schemes $G$ and $G'$ over $k$, the $k$-superfunctor
\[
\bfHom(G,G') : \csalg_k \rightarrow \mathfrak{sets}
\]
is defined by
\[
\bfHom(G,G')(A) = \Hom_{Grp/A}(G \otimes_k A, G' \otimes_k A),
\]
the set of $A$-supergroup scheme homomorphisms $\rho: G \otimes_k A \rightarrow G' \otimes_k A$. In \cite[Lemma 3.3.2]{Drupieski:2017a} we calculated $\bfHom(G,G')(A)$ in the special case $G = G' = \Mrs$ under the assumption that $A$ is purely even and reduced. In this section we extend those calculations to the case when $G'$ is an arbitrary height-$r$ infinitesimal elementary supergroup scheme and $G = \Mr$.

Let $s \geq 1$ and let $\eta \in k$. As observed in \cite[Lemma 3.1.7]{Drupieski:2017a}, the group algebras $k\Mrseta$ and $k\Mrs$ are isomorphic as $k$-supercoalgebras, and hence $k[\Mrseta]$ and $k[\Mrs]$ are isomorphic as $k$-super\-algebras. Under this identification, the coproduct on $k[\M_{1;s,\eta}]$ is identified in \cite[Lemma 3.1.9]{Drupieski:2017a}. More generally, for $r \geq 2$ one has:

\begin{lemma} \label{lem:k[Mrseta]coproduct}
Let $r \geq 2$, let $s \geq 1$, and let $\eta \in k$. Then identifying $k[\Mrseta]$ and $k[\Mrs]$ as $k$-super\-algebras, the coproduct on $k[\Mrseta]$ satisfies the formulas
	\begin{align*}
	\Delta(\tau) &= \tau \otimes 1 + 1 \otimes \tau, \\
	\Delta(\sigma_\ell) &= \textstyle \sum_{i+j = \ell} \sigma_i \otimes \sigma_j + \sum_{i+j+p=\ell} \sigma_i \tau \otimes \sigma_j \tau, & \text{and} \\
	\Delta(\theta) &= \textstyle \theta \otimes 1 + 1 \otimes \theta - \eta \cdot \left( \sum_{i=1}^{p^s-1} \sigma_i \otimes \sigma_{p^s-i} + \sum_{i+j+p=p^s} \sigma_i \tau \otimes \sigma_j \tau \right).
	\end{align*}
\end{lemma}

\begin{proof}
The stated formulas can be deduced via duality from the algebra structure of the group algebra $k\Mrseta$ in the same manner as the proof of \cite[Lemma 3.1.9]{Drupieski:2017a}.
\end{proof}

Portions of the next proposition are contained already in \cite[Lemma 3.3.2]{Drupieski:2017a}. Recall from \cite[\S3.1]{Drupieski:2017a} that $k[\Gar] = k[\theta]/\subgrp{\theta^{p^r}}$ and $k[\Gam] = \Lambda(\tau)$. Also recall for $t \geq 1$ that the canonical quotient maps $\Mr \twoheadrightarrow \Mrt \twoheadrightarrow \Gar$ are defined via the subalgebra inclusions $k[\Gar] \hookrightarrow k[\Mrt] \hookrightarrow k[\Mr]$, and similarly for the other quotient maps referenced in the next proposition.

\begin{proposition} \label{prop:identifyhom}
Let $A = \Azero \in \calg_k$ be a purely even commutative $k$-algebra.
	\begin{enumerate}
	\item \label{item:HomMrone} Let $t \geq 1$. Then the canonical quotient maps $\Mr \twoheadrightarrow \Mrt \twoheadrightarrow \Mrone$ induce identifications
		\begin{align*}
		\bfHom(\Mrone,\Mrone)(A) &= \bfHom(\Mrt,\Mrone)(A) = \bfHom(\Mr,\Mrone)(A) \\
		&= \set{ (\mu,a_0,\ldots,a_{r-1}) \in A^{r+1}}.
		\end{align*}
	Given a tuple $(\mu,a_0,\ldots,a_{r-1}) \in A^{r+1}$, the corresponding comorphism $\phi: A[\Mrone] \rightarrow A[\Mr]$ is completely specified by the formulas $\phi(\tau) = \mu \cdot \tau$ and $\phi(\theta) = \sum_{i=0}^{r-1} a_i \cdot \theta^{p^i}$.
	
	\item \label{item:HomMrs} Let $2 \leq s \leq t$, and suppose $A$ is reduced. Then the canonical quotient maps $\Mr \twoheadrightarrow \Mrt \twoheadrightarrow \Mrs$ induce identifications
		\begin{align*}
		\bfHom(\Mrs,\Mrs)(A) &= \bfHom(\Mrt,\Mrs)(A) = \bfHom(\Mr,\Mrs)(A) \\
		&= \set{ (\mu,a_0,\ldots,a_{r-1},b_s) \in A^{r+2} : \mu^2 = a_0^{p^r}}.
		\end{align*}
	Given a tuple $(\mu,a_0,\ldots,a_{r-1},b_s) \in A^{r+2}$ such that $\mu^2 = a_0^{p^r}$, the corresponding comorphism $\phi: A[\Mrs] \rightarrow A[\Mr]$ is completely specified by the formulas $\phi(\tau) = \mu \cdot \tau$, $\phi(\theta) = \sum_{i=0}^{r-1} a_i \cdot \theta^{p^i}$, $\phi(\sigma_i) = a_0^{ip^{r-1}} \cdot \sigma_i$ for $0 \leq i < p^{s-1}$, and $\phi(\sigma_{p^{s-1}}) = a_0^{p^{r+s-2}} \cdot \sigma_{p^{s-1}} + b_s \cdot \sigma_1$.
	
	\item \label{item:HomMrseta} Suppose $r \geq 2$, let $1 \leq s < t$, and let $0 \neq \eta \in k$. Suppose $A$ is reduced. Then the canonical quotient maps $\Mr \twoheadrightarrow \Mrt \twoheadrightarrow \M_{r;s+1}$ induce identifications
		\begin{align*}
		\bfHom(\M_{r;s+1},\Mrseta)(A) &= \bfHom(\Mrt,\Mrseta)(A) = \bfHom(\Mr,\Mrseta)(A) \\
		&= \set{ (\mu,a_0,\ldots,a_{r-1}) \in A^{r+1} : \mu^2 = a_0^{p^r} }
		\end{align*}
	Given a tuple $(\mu,a_0,\ldots,a_{r-1}) \in A^{r+1}$ such that $\mu^2 = a_0^{p^r}$, the corresponding comorphism $\phi: A[\Mrseta] \rightarrow A[\Mr]$ is completely specified by the formulas $\phi(\tau) = \mu \cdot \tau$,
		\[ \textstyle
		\phi(\theta) = (\sum_{i=0}^{r-1} a_i \cdot \theta^{p^i}) - \eta \cdot (a_0^{p^{r+s-1}} \sigma_{p^s}),
		\]
	and $\phi(\sigma_i) = a_0^{ip^{r-1}} \cdot \sigma_i$ for $0 \leq i < p^s$.

	\item \label{item:HomGar} Let $t \geq 1$. Then the canonical quotient maps $\Mr \twoheadrightarrow \Mrt \twoheadrightarrow \Gar$ induce identifications
		\[
		\bfHom(\Gar,\Gar)(A) = \bfHom(\Mrt,\Gar)(A) = \bfHom(\Mr,\Gar)(A) = A^r
		\]
	Given a tuple $(a_0,\ldots,a_{r-1}) \in A^r$, the corresponding comorphism $\phi: A[\Gar] \rightarrow A[\Mr]$ is completely specified by the formula $\phi(\theta) = \sum_{i=0}^{r-1} a_i \cdot \theta^{p^i}$.

	\item \label{item:HomGaminus} Let $t \geq 1$. Then the canonical quotient maps $\Mr \twoheadrightarrow \Mrt \twoheadrightarrow \Gam$ induce identifications
		\[
		\bfHom(\Gam,\Gam)(A) = \bfHom(\Mrt,\Gam)(A) = \bfHom(\Mr,\Gam)(A) = A			\]
	Given $\mu \in A$, the corresponding comorphism $\phi: A[\Gam] \rightarrow A[\Mr]$ is specified by $\phi(\tau) = \mu \cdot \tau$.
	\end{enumerate}
\end{proposition}

\begin{proof}
Recall that specifying a homomorphism of $A$-supergroup schemes $\rho: G \rightarrow G'$ is equivalent to specifying a homomorphism of $A$-Hopf superalgebras $\rho^*: A[G'] \rightarrow A[G]$. The proof strategy for each of the statements is to first identify the set of all Hopf superalgebra homomorphisms from $A[\Mrone]$ (resp.\ from $A[\Mrs]$, $A[\Mrseta]$, $A[\Gar]$, $A[\Gam]$) into $A[\Mr]$. Using the explicit identifications, one can then observe that the image of each Hopf superalgebra homomorphism is contained in the Hopf subalgebra $A[\Mrone]$ (resp.\ $A[\Mrs]$, $A[\M_{r;s+1}]$, $A[\Gar]$, $A[\Gam]$) of $A[\Mr]$, and from that observation the first two equalities in each statement follow immediately. The identification of all homomorphisms $\phi: A[\Mrone] \rightarrow A[\Mr]$ follows from a word-for-word repetition of the argument in the first paragraph of the proof of \cite[Lemma 3.3.2]{Drupieski:2017a}. The proofs of statements \eqref{item:HomGar} and \eqref{item:HomGaminus} also follow via reasoning entirely similar to that in the proof of \cite[Lemma 3.3.2]{Drupieski:2017a}, so for the remainder of this proof we focus on statements \eqref{item:HomMrs} and \eqref{item:HomMrseta}. We first prove \eqref{item:HomMrseta} and then indicate how the argument should be modified for \eqref{item:HomMrs}.

Suppose $r \geq 2$, let $s \geq 1$, and let $0 \neq \eta \in k$. Let $\phi: A[\Mrseta] \rightarrow A[\Mr]$ be a homomorphism of $A$-Hopf superalgebras, identifying the Hopf structure on $A[\Mrseta]$ as in Lemma \ref{lem:k[Mrseta]coproduct}. Our first goal is to describe the action of $\phi$ on the elements $\tau,\sigma_1,\sigma_p,\ldots,\sigma_{p^{s-1}} \in A[\Mrseta]$. Since $\phi$ is by definition an even map, it must map $\tau$ to an odd primitive element in $A[\Mr]$. Since $A$ is purely even by assumption, this implies that $\phi(\tau) = \mu \cdot \tau$ for some $\mu \in A$. Next, by the assumption that $r \geq 2$, the element $\sigma_1 \in A[\Mrseta]$ is primitive and $(\sigma_1)^p = 0$.\footnote{If $\eta \neq 0$ and $r=1$, then $\sigma_1$ need not be primitive in $A[\Mrseta]$; cf.\ \cite[Lemma 3.1.9]{Drupieski:2017a}.} Then $\phi(\sigma_1)$ must be a primitive element in $A[\Mr]$ such that $\phi(\sigma_1)^p = 0$, so $\phi(\sigma_1) = a \cdot \sigma_1$ for some $a \in A$. Since $(\sigma_1)^n = n! \cdot \sigma_n$, this further implies for $0 \leq i < p$ that $\phi(\sigma_i) = a^i \cdot \sigma_i$.

Our goal is to show for $0 \leq i < p^{s-1}$ that $\phi(\sigma_i) = a^i \cdot \sigma_i$, and that $\phi(\sigma_{p^{s-1}}) = a^{p^{s-1}} \cdot \sigma_{p^{s-1}} + b_s \cdot \sigma_1$ for some $b_s \in A$. Taking $b_1 = 0$, the claim is true by the results of the preceding paragraph in the case $s=1$, \emph{so assume for the rest of this paragraph and the next that $s \geq 2$}. First consider $\phi(\sigma_p)$. Since $\phi$ is a Hopf superalgebra homomorphism, then
\begin{equation} \label{eq:coproductsigmap}
\begin{split}
\Delta \circ \phi(\sigma_p) &= (\phi \otimes \phi) \circ \Delta(\sigma_p) \\
&= \textstyle \phi(\sigma_p) \otimes 1 + 1 \otimes \phi(\sigma_p) + \left(\sum_{i=1}^{p-1} \phi(\sigma_i) \otimes \phi(\sigma_{p-i})\right) + \phi(\tau) \otimes \phi(\tau) \\
&= \textstyle \phi(\sigma_p) \otimes 1 + 1 \otimes \phi(\sigma_p) + a^p \cdot \left(\sum_{i=1}^{p-1} \sigma_i \otimes \sigma_{p-i} \right) + \mu^2 \cdot ( \tau \otimes \tau ).
\end{split}
\end{equation}
By \cite[Remark 3.1.3(3)]{Drupieski:2017a}, $A[\Mr]$ is a $\Z$-graded Hopf superalgebra with $\deg(\theta) = 2$, $\deg(\sigma_j) = 2jp^{r-1}$, and $\deg(\tau) = p^r$. In particular, the degree-$2p^r$ component of $A[\Mr]$ is spanned by $\sigma_p$. Since the coproduct on a Hopf (super)algebra is automatically injective, \eqref{eq:coproductsigmap} then implies that the graded components of $\phi(\sigma_p)$ of $\Z$-degrees other than $2p^r$ must be primitive. Denoting the degree-$2p^r$ component of $\phi(\sigma_p)$ by $c \cdot \sigma_p$, \eqref{eq:coproductsigmap} says that
	\[ \textstyle
	a^p \cdot \left(\sum_{i=1}^{p-1} \sigma_i \otimes \sigma_{p-i} \right) + \mu^2 \cdot ( \tau \otimes \tau ) = \Delta(c \cdot \sigma_p) = c \cdot \left(\sum_{i=1}^{p-1} \sigma_i \otimes \sigma_{p-i} + \tau \otimes \tau \right),
	\]
and hence $a^p = c = \mu^2$. The primitive components of $\phi(\sigma_p)$ must be $p$-nilpotent because $(\sigma_p)^p = 0$. By the assumption that $A$ is reduced, the only $p$-nilpotent primitive elements in $A[\Mr]$ are scalar multiples of $\sigma_1$. Then $\phi(\sigma_p) = a^p \cdot \sigma_p + c_1 \cdot \sigma_1$ for some $c_1 \in A$.

Now let $1 \leq n < s-1$. Suppose by way of induction that $\phi(\sigma_i) = a^i \cdot \sigma_i$ for $0 \leq i < p^n$, and that $\phi(\sigma_{p^n}) = a^{p^n} \cdot \sigma_{p^n} + c_n \cdot \sigma_1$ for some $c_n \in A$. We will then show that $c_n=0$, that $\phi(\sigma_i) = a^i \cdot \sigma_i$ for $0 \leq i < p^{n+1}$, and that $\phi(\sigma_{p^{n+1}}) = a^{p^{n+1}} \cdot \sigma_{p^{n+1}} + c_{n+1} \cdot \sigma_1$ for some $c_{n+1} \in A$; by induction, this will complete the goal laid out at the beginning of the preceding paragraph. From the divided power relation $\sigma_i \sigma_j = \binom{i+j}{i} \sigma_{i+j}$ in $A[\Mr]$ (and Lucas's Theorem for binomial coefficients mod $p$), it follows that if $0 \leq i < p^n$ and $0 \leq j < p$, then
	\[
	\sigma_{i+jp^n} = \sigma_i \cdot \sigma_{jp^n} = \frac{1}{j!} \cdot \sigma_i \cdot (\sigma_{p^n})^j.
	\]
and hence
	\begin{equation} \label{eq:phisigmaijpn}
	\phi(\sigma_{i+jp^n}) = \frac{1}{j!} (a^i \cdot \sigma_i) \left( \sum_{\ell=0}^j \binom{j}{\ell} a^{\ell p^n} c_n^{j-\ell} \cdot (\sigma_{p^n})^\ell (\sigma_1)^{j-\ell} \right) = \sum_{\ell=0}^j a^{i+\ell p^n} c_n^{j-\ell} \cdot \sigma_{i+\ell p^n} \cdot \sigma_{j-\ell}.
	\end{equation}
Since $\phi$ is a Hopf superalgebra homomorphism, then
\begin{multline} \label{eq:pn1coproduct}
\Delta \circ \phi(\sigma_{p^{n+1}}) = (\phi \otimes \phi) \circ \Delta(\sigma_{p^{n+1}}) \\
= \phi(\sigma_{p^{n+1}}) \otimes 1 + 1 \otimes \phi(\sigma_{p^{n+1}}) + \sum_{i=1}^{p^{n+1}-1} \phi(\sigma_i) \otimes \phi(\sigma_{p^{n+1}-i}) +  \sum_{i+j+p=p^{n+1}} \mu^2 \cdot \phi(\sigma_i) \tau \otimes \phi(\sigma_j) \tau.
\end{multline}
Using \eqref{eq:phisigmaijpn}, and considering the $\Z$-grading on $A[\Mr]$, one can check that the homogeneous component of lowest $\Z$-degree in the sum $\sum_{i=1}^{p^{n+1}-1} \phi(\sigma_i) \otimes \phi(\sigma_{p^{n+1}-1})$ is
	\[
	(c_n)^p \cdot \left( \sigma_1 \otimes \sigma_{p-1} + \sigma_2 \otimes \sigma_{p-2} + \cdots + \sigma_{p-1} \otimes \sigma_1 \right).
	\]
By the injectivity of the coproduct, this implies that the degree-$2p^r$ component of $\phi(\sigma_{p^{n+1}})$ must be $(c_n)^p \cdot \sigma_p$. On the other hand, one can check that no scalar multiple of $\tau \otimes \tau$ occurs in the expansion of the sum $\sum_{i+j+p=p^{n+1}} \phi(\sigma_i) \tau \otimes \phi(\sigma_j) \tau$. This implies that the degree-$2p^r$ component of $\phi(\sigma_{p^{n+1}})$ must be $0$, and hence because $A$ is reduced implies that $c_n=0$. So $\phi(\sigma_{p^n})$ is simply equal to $a^{p^n} \cdot \sigma_{p^n}$, and it follows from the algebra relations among $\sigma_1,\sigma_p,\ldots,\sigma_{p^n}$ (cf.\ \cite[(3.1.1)]{Drupieski:2017a}) that $\phi(\sigma_i) = a^i \cdot \sigma_i$ for $0 \leq i < p^{n+1}$. Now, considering \eqref{eq:pn1coproduct} and arguing as in the previous paragraph for $\phi(\sigma_p)$, it follows that $\phi(\sigma_{p^{n+1}}) = a^{p^{n+1}} \cdot \sigma_{p^{n+1}} + c_{n+1} \cdot \sigma_1$ for some $c_{n+1} \in A$.

The inductive argument of the previous two paragraphs shows that $\phi(\sigma_i) = a^i \cdot \sigma_i$ for $0 \leq i < p^{s-1}$, and that $\phi(\sigma_{p^{s-1}}) = a^{p^{s-1}} \cdot \sigma_{p^{s-1}} + b_s \cdot \sigma_1$ for some $b_s \in A$, with $b_1 = 0$. To finish describing the Hopf superalgebra homomorphism $\phi: A[\Mrseta] \rightarrow A[\Mr]$, it remains to describe the action of $\phi$ on the generator $\theta \in A[\Mrseta]$. From the coproduct formula of Lemma \ref{lem:k[Mrseta]coproduct} one gets
	\begin{multline} \label{eq:coproducttheta}
	\Delta \circ \phi(\theta) = (\phi \otimes \phi) \circ \Delta(\theta) \\
	= \phi(\theta) \otimes 1 + 1 \otimes \phi(\theta) - \eta \cdot \left( \sum_{i=1}^{p^s-1} \phi(\sigma_i) \otimes \phi(\sigma_{p^s-i}) + \sum_{i+j+p=p^s} \mu^2 \cdot \phi(\sigma_i) \tau \otimes \phi(\sigma_j) \tau \right).
	\end{multline}
Applying \eqref{eq:phisigmaijpn}, and considering the $\Z$-grading on $A[\Mr]$, it follows as in the third paragraph of the proof that the $\Z$-graded components of $\phi(\theta)$ of degrees less than $2p^r$ must be primitive, and hence must be an $A$-linear combination of the elements $\theta,\theta^p,\ldots,\theta^{p^{r-1}}$. One can further check, as in the previous paragraph, that the component in $\sum_{i=1}^{p^s-1} \phi(\sigma_i) \otimes \phi(\sigma_{p^s-i})$ of lowest $\Z$-degree is $(b_s)^p \cdot (\sum_{i=1}^{p-1} \sigma_i \otimes \sigma_{p-i})$, but that no scalar multiple of $\tau \otimes \tau$ occurs in $\sum_{i+j+p=p^s} \phi(\sigma_i) \tau \otimes \phi(\sigma_j)\tau$. Since $\eta \neq 0$ and since $A$ is reduced, this implies as in the previous paragraph that $b_s =0$. Then $\phi(\sigma_{p^{s-1}})$ is simply equal to $a^{p^{s-1}} \cdot \sigma_{p^{s-1}}$, so it follows for all $0 \leq i < p^s$ that $\phi(\sigma_i) = a^i \cdot \sigma_i$.

Now if $s \geq 2$, then $\mu^2 = a^p$ by the second paragraph of the proof, so we get
	\[
	\Delta \circ \phi(\theta) = \phi(\theta) \otimes 1 + 1 \otimes \phi(\theta) - (\eta \cdot a^{p^s}) \cdot \left( \sum_{i=1}^{p^s-1} \sigma_i \otimes \sigma_{p^s-i} + \sum_{i+j+p=p^s} \sigma_i \tau \otimes \sigma_j \tau \right),
	\]
and hence $\phi(\theta) = (\sum_{i=0}^{r-1} a_i \cdot \theta^{p^i}) - (\eta \cdot a^{p^s}) \cdot \sigma_{p^s}$ for some $a_0,a_1,\ldots,a_{r-1} \in A$ by the injectivity of the coproduct. But
	\begin{equation} \label{eq:sigma1theta}
	a \cdot \sigma_1 = \phi(\sigma_1) = \phi(\theta^{p^{r-1}}) = \phi(\theta)^{p^{r-1}} = a_0^{p^{r-1}} \cdot \theta^{p^{r-1}} = a_0^{p^{r-1}} \cdot \sigma_1,
	\end{equation}
so $a = a_0^{p^{r-1}}$. Thus for $s \geq 2$, the homomorphism $\phi$ is completely determined by the elements $\mu,a_0,a_1,\ldots,a_{r-1} \in A$, and these elements satisfy the additional condition that $\mu^2 = a^p = a_0^{p^r}$. On the other hand, if $s = 1$ then we know from the second paragraph of the proof that $\phi(\sigma_i) = a^i \cdot \sigma_i$ for $0 \leq i < p$, and hence
	\begin{equation} \label{eq:coproductthetas=1}
	\Delta \circ \phi(\theta) = \phi(\theta) \otimes 1 + 1 \otimes \phi(\theta) - \eta \cdot \left( \sum_{i=1}^{p-1} a^p \cdot (\sigma_i \otimes \sigma_{p-i}) + \mu^2 \cdot (\tau \otimes \tau) \right).
	\end{equation}
Now arguing as we did for $\phi(\sigma_p)$ in the third paragraph of the proof, and using the fact that $\eta \neq 0$, it follows that $a^p = \mu^2$ and $\phi(\theta) = (\sum_{i=0}^{r-1} a_i \cdot \theta^{p^i} ) - (\eta \cdot a^p) \cdot \sigma_p$ for some scalars $a_0,\ldots,a_{r-1} \in A$. Then $a_0^{p^{r-1}} = a$ by \eqref{eq:sigma1theta}, and hence $\mu^2 = a^p = a_0^{p^r}$. Thus for all $s \geq 1$, the Hopf superalgebra homomorphism $\phi: A[\Mrseta] \rightarrow A[\Mr]$ is completely specified by the elements $\mu,a_0,\ldots,a_{r-1} \in A$, which satisfy the additional condition that $\mu^2 = a_0^{p^r}$. Conversely, given any elements $\mu,a_0,\ldots,a_{r-1} \in A$ such that $\mu^2 = a_0^{p^r}$, the assignments
	\[ \textstyle
	\tau \mapsto \mu \cdot \tau, \quad \theta \mapsto (\sum_{i=0}^{r-1} a_i \cdot \theta^{p^i}) - (\eta \cdot a_0^{p^{r+s-1}}) \cdot \sigma_{p^s}, \quad \text{and} \quad \sigma_i \mapsto a_0^{i p^{r-1}} \cdot \sigma_i
	\]
extend to a homomorphism of Hopf superalgebras $\phi: A[\Mrseta] \rightarrow A[\Mr]$. This completes the proof that $\bfHom(\Mr,\Mrseta)(A)$ identifies as a set with $\{(\mu,a_0,\ldots,a_{r-1}) \in A^{r+1}: \mu^2 = a_0^{p^r} \}$.

For the proof of \eqref{item:HomMrs}, one first proceeds as in the second, third, and fourth paragraphs to show that $\phi(\tau) = \mu \cdot \tau$, $\phi(\sigma_i) = a^i \cdot \sigma_i$ for $0 \leq i < p^{s-1}$, and $\phi(\sigma_{p^{s-1}}) = a^{p^{s-1}} \cdot \sigma_{p^{s-1}} + b_s \cdot \sigma_1$, for some elements $\mu,a,b_s \in A$ with $\mu^2 = a^p$. Then one argues as in the fifth and sixth paragraphs (with $\eta = 0$) to show that $\phi(\theta) = \sum_{i=0}^{r-1} a_i \cdot \theta^{p^i}$ for some $a_0,\ldots,a_{r-1} \in A$ with $a_0^{p^{r-1}} = a$. So $\phi: A[\Mrs] \rightarrow A[\Mr]$ is completely specified by the elements $\mu,a_0,\ldots,a_{r-1},b_s \in A$, which satisfy the additional condition that $\mu^2 = a_0^{p^r}$. Conversely, given $s \geq 2$ and given any choice of elements $\mu,a_0,\ldots,a_{r-1},b_s \in A$ such that $\mu^2 = a_0^{p^r}$, one can check that the assignments
	\begin{align*}
	\tau &\mapsto \mu \cdot \tau, & \sigma_i &\mapsto a_0^{i p^{r-1}} \cdot \sigma_i \quad \text{for $0 \leq i < p^{s-1}$}, \\
	\theta &\mapsto a_0 \cdot \theta + a_1 \cdot \theta^p + \cdots + a_{r-1} \cdot \theta^{p^{r-1}}, & \sigma_{p^{s-1}} &\mapsto a_0^{p^{r+s-2}} \cdot \sigma_{p^{s-1}} + b_s \cdot \sigma_1,
	\end{align*}
uniquely specify a Hopf superalgebra homomorphism $\phi: A[\Mrs] \rightarrow A[\Mr]$.
\end{proof}

\begin{remark} \  \label{rem:Hom}
\begin{enumerate}
\item \label{item:quotientmaps} Let $s \geq 1$. The canonical quotient homomorphisms $\pi_{r;s} : \Mr \rightarrow \Mrs$, $q: \Mr \rightarrow \Gar$, and $q^-: \Mr \rightarrow \Gam$ are labeled by $(1,1,0,\ldots,0)$, $(1,0,\ldots,0)$, and $\mu = 1$, respectively. If $r \geq 2$ and $\eta \in k$, the canonical quotient map $\pi_{r;s,\eta}: \Mr \twoheadrightarrow \Mrseta$ is labeled by $(1,1,0,\ldots,0)$.

\item \label{item:iota} Suppose $s \geq 2$, and let $\rho = \rho_{(0,0,1)} : \Mone \rightarrow \Mones$ be the homo\-morphism labeled by the tuple $(0,0,1) \in k^3$. The comorphism $\rho^*: k[\Mones] \rightarrow k[\Mone]$ satisfies $\rho^*(\tau) = 0$, $\rho^*(\sigma_i) = 0$ if $i$ is not divisible by $p^{s-1}$, and $\rho^*(\sigma_{j p^{s-1}}) = \sigma_j$ for $0 \leq j < p$. In particular, $\rho^*$ factors through the subalgebra $k[\sigma_1]/\subgrp{\sigma_1^p} \cong k[\Gaone]$ of $k[\Mone]$ (note that $\theta = \theta^{p^{r-1}} = \sigma_1$ in $k[\Mone]$), so
	\begin{gather*}
	\rho^* = q^* \circ \iota^*: k[\Mones] \twoheadrightarrow k[\Gaone] \hookrightarrow k[\Mone], \\
	\text{and} \quad \rho = \iota \circ q: \Mone \twoheadrightarrow \Gaone \hookrightarrow \Mones,
	\end{gather*}
where $\iota^*: k[\Mones] \twoheadrightarrow k[\Gaone]$ is the Hopf superalgebra homomorphism obtained by restricting the codomain of $\rho^*$, and $\iota: \Gaone \hookrightarrow \Mones$ is the corresponding map of supergroup schemes. The map $\iota$ is a closed embedding with $k[\im(\iota)] = k[\Mones]/\subgrp{\tau,\sigma_1,\ldots,\sigma_{p^{s-1}-1}}$. Since $\tau$ is an element of the defining ideal of $\im(\iota)$, it follows (e.g., using the formula \cite[(3.1.2)]{Drupieski:2017a} for the group multiplication in $\Mones$) that $\im(\iota)$ is central and hence normal in $\Mones$, and that $\Mones/\im(\iota) \cong \M_{1;s-1}$.

\item \label{item:1001factorizations} Recall that $\Moneone = \Gaone \times \Gam$. Then in a similar fashion to the previous item, the maps $\phi_{(0,1)}: \Mone \rightarrow \Moneone$ and $\phi_{(1,0)}: \Mone \rightarrow \Moneone$ labeled by $(0,1)$ and $(1,0)$, respectively, admit factorizations $\phi_{(0,1)} : \Mone \twoheadrightarrow \Gaone \hookrightarrow \Moneone$ and $\phi_{(1,0)} : \Mone \twoheadrightarrow \Gam \hookrightarrow \Moneone$, in which the first arrow in each composition is the canonical quotient map, and the second arrow is the obvious map arising from the direct product decomposition $\Moneone = \Gaone \times \Gam$.
\end{enumerate}
\end{remark}

\begin{corollary} \label{cor:HomMr}
Let $A = \Azero \in \calg_k$ be a purely even commutative $k$-algebra, and suppose $A$ is reduced. Then there exists a natural identification
	\[
	\bfHom(\Mr,\Mr)(A) = \set{ (\mu,a_0,\ldots,a_{r-1}) \in A^{r+1} : \mu^2 = a_0^{p^r}}.
	\]
Under this identification, the tuple $(\mu,\ualpha) := (\mu,a_0,\ldots,a_{r-1})$ corresponds to the supergroup homomorphism whose comorphism $\phi = \phi_{(\mu,\ualpha)} : A[\Mr] \rightarrow A[\Mr]$ satisfies
	\[ \textstyle
	\phi(\tau) = \mu \cdot \tau, \quad \phi(\theta) = \sum_{i=0}^{r-1} a_i \cdot \theta^{p^i}, \quad \text{and} \quad \phi(\sigma_i) = a_0^{ip^{r-1}} \cdot \sigma_i \quad \text{for $i \geq 0$.}
	\]
\end{corollary}

\begin{proof}
As in the preceding proof, we first classify all Hopf superalgebra maps $\phi: A[\Mr] \rightarrow A[\Mr]$. Let $\phi$ be such a map. Then $\phi$ restricts for each $s \geq 1$ to a homomorphism $\phi_s: A[\Mrs] \rightarrow A[\Mr]$. Since the restriction of $\phi_{s+1}$ to $A[\Mrs]$ must be the same as $\phi_s$, it follows that $\phi = \phi_{(\mu,\ualpha)}$ for some tuple $(\mu,a_0,\ldots,a_{r-1}) \in A^{r+1}$ such that $\mu^2 = a_0^{p^r}$. Conversely, given any tuple $(\mu,\ualpha) \in A^{r+1}$ such that $\mu^2 = a_0^{p^r}$, one can check that the formulas
	\[ \textstyle
	\phi(\tau) = \mu \cdot \tau, \quad \phi(\theta) = \sum_{i=0}^{r-1} a_i \cdot \theta^{p^i}, \quad \text{and} \quad \phi(\sigma_i) = a_0^{ip^{r-1}} \cdot \sigma_i \quad \text{for $i \geq 0$}
	\]
uniquely specify a homomorphism of $A$-Hopf superalgebras $\phi_{(\mu,\ualpha)}: A[\Mr] \rightarrow A[\Mr]$.
\end{proof}

Let $\bfAut(G) \subset \bfHom(G,G)$ be the subfunctor of automorphisms.

\begin{lemma} \label{lemma:automorphisms}
Under the identifications of Proposition \ref{prop:identifyhom} and Corollary \ref{cor:HomMr},
		\begin{align*}
		\bfAut(\Mrone)(k) &= \set{ (\mu,a_0,\ldots,a_{r-1}) \in k^{r+1} : \mu \neq 0 \text{ and } a_0 \neq 0}, \\
		\bfAut(\Mrs)(k) &= \set{ (\mu,a_0,\ldots,a_{r-1},b_s) \in k^{r+2} : \mu^2 = a_0^{p^r} \neq 0} & \text{if $s \geq 2$,} \\
		\bfAut(\Gar)(k) &= \set{ (a_0,\ldots,a_{r-1}) \in k^r : a_0 \neq 0}, \\
		\bfAut(\Gam)(k) &= k - \set{0}, & \text{and} \\
		\bfAut(\Mr)(k) &= \set{ (\mu,a_0,\ldots,a_{r-1}) \in k^{r+1} : \mu^2 = a_0^{p^r} \neq 0}.
		\end{align*}
\end{lemma}

\begin{proof}
We sketch the details for the calculation of $\bfAut(\Mrs)(k)$; the other calculations are similar and/or easier. First, the condition $\mu^2 = a_0^{p^r} \neq 0$ is evidently necessary in order for the corresponding homomorphism to be invertible, since otherwise the comorphism $\phi: k[\Mrs] \rightarrow k[\Mrs]$ has a nontrivial kernel. For the reverse set inclusion, let $(\mu,a_0,\ldots,a_{r-1},b_s) \in k^{r+2}$ such that $\mu^2 = a_0^{p^r} \neq 0$, and let $\phi: k[\Mrs] \rightarrow k[\Mrs]$ be the corresponding comorphism. Since $k[\Mrs]$ is a finite-dimensional $k$-vector space, to show that $\phi$ is invertible it suffices to show that $\phi$ admits a left inverse. For $1 \leq i \leq r-1$, let $\phi_i: k[\Mrs] \rightarrow k[\Mrs]$ be the homomorphism labeled by
	\[
	(1,1,0,\ldots,-a_0^{-1}a_i,0,\ldots,0) \in k^{r+2},
	\]
where $-a_0^{-1}a_i$ appears in the $(i+1)$-st coordinate. Now by induction on $i$, one can check that
	\[
	\phi_i \circ \phi_{i-1} \circ \cdots \circ \phi_1 \circ \phi : k[\Mrs] \rightarrow k[\Mrs]
	\]
is labeled by a tuple of the form $(\mu,a_0,0,\ldots,0,a_{i+1}',\ldots,a_{r-1}',b_s')$ for some $a_{i+1}',\ldots,a_{r-1}',b' \in k$. In particular, $\phi_{r-1} \circ \cdots \circ \phi_1 \circ \phi$ is labeled by $(\mu,a_0,0,\ldots,0,b')$ for some $b' \in k$. Next let $\phi_0: k[\Mrs] \rightarrow k[\Mrs]$ be the map labeled by $(\mu^{-1},a_0^{-1},0,\ldots,0)$. Then one can check that $\phi_0 \circ \phi_{r-1} \circ \cdots \circ \phi_1 \circ \phi$ is labeled by $(1,1,0,\ldots,0,b'a_0^{-p^{r-1}})$. Finally, let $\psi: k[\Mrs] \rightarrow k[\Mrs]$ be the map labeled by $(1,1,0,\ldots,0,-b'a_0^{-r^{r-1}})$. Then $\psi \circ \phi_0 \circ \phi_{r-1} \circ \cdots \circ \phi_1 \circ \phi$ is labeled by $(1,1,0,\ldots,0)$, and hence is the identity map. Thus, $\psi \circ \phi_0 \circ \phi_{r-1} \circ \cdots \circ \phi_1$ is a left inverse for $\phi$.
\end{proof}

\subsection{Commuting nilpotent supermatrices} \label{subsection:nilpotentsupermatrices}

Given nonnegative integers $m$ and $n$, the $k$-super\-functor $\bsV_r(\GLmn) : \csalg_k \rightarrow \mathfrak{sets}$ is defined by
\begin{align*}
\bsV_r(\GLmn)(A) = \Big \{ (\alpha_0,\ldots,\alpha_{r-1},\beta) &\in (\Matmn(A)_{\zero})^{\times r} \times \Matmn(A)_{\one} : \\
&[\alpha_i,\alpha_j] = [\alpha_i,\beta] = 0 \text{ for all $0 \leq i,j \leq r-1$}, \\
&\alpha_i^p = 0 \text{ for all $0 \leq i \leq r-2$}, \text{ and } \alpha_{r-1}^p + \beta^2 = 0 \Big \}.
\end{align*}

\begin{definition}
For $m,n \in \N$, let $\bsNr(\GLmn)$ be the subfunctor of $\bsV_r(\GLmn)$ defined by
	\[
	\bsNr(\GLmn)(A) = \set{ (\alpha_0,\ldots,\alpha_{r-1},\beta) \in \bsV_r(\GLmn)(A) : \alpha_{r-1} \text{ is nilpotent}}.
	\]
\end{definition}

Given a commutative superalgebra $A \in \csalg_k$, let
	\[
	A\Mr = \Hom_A(A[\Mr],A) = \Hom_A( k[\Mr] \otimes_k A,A) \cong \Hom_k(k[\Mr],A)
	\]
be the group algebra of the $A$-supergroup scheme $\Mr \otimes_k A$. Then $A\Mr$ admits the identification
	\[
	A\Mr = A[[u_0,\ldots,u_{r-1},v]]/\subgrp{u_0^p,\ldots,u_{r-2}^p,u_{r-1}^p+v^2}.
	\]
The next result was stated, with a different proof, in \cite[Proposition 3.3.5]{Drupieski:2017a}. The proof given below is in the spirit of the proof of \cite[Proposition 1.2]{Suslin:1997}. The upshot of the new proof is that we get, in Remark \ref{remark:exp}, a description of homomorphisms in terms of exponential maps entirely analogous to the classical description given in \cite[Remark 1.3]{Suslin:1997a}.

\begin{proposition} \label{proposition:HomMrGLmn}
Let $m,n \in \N$. Then for each $A \in \csalg_k$, there exists a natural identification
\[
\bsNr(\GLmn)(A) = \bfHom(\Mr,\GLmn)(A), \quad (\ualpha|\beta) \mapsto \rho_{(\ualpha|\beta)}.
\]
Given $(\ualpha|\beta) := (\alpha_0,\ldots,\alpha_{r-1},\beta) \in \bsN_r(\GLmn)(A)$, the corresponding supergroup homomorphism $\rho_{(\ualpha|\beta)}: \Mr \otimes_k A \rightarrow \GLmn \otimes_k A$ is such that the induced action of the group algebra $A\Mr$ on $A^{m|n}$ is defined by having the generator $u_i \in k\Mr \subset A\Mr$ act via the matrix $\alpha_i \in \Matmn(A)_{\zero}$, and by having the generator $v \in k\Mr \subset A\Mr$ act via the matrix $\beta \in \Matmn(A)_{\one}$.
\end{proposition}

\begin{proof}
Specifying a homomorphism of $A$-supergroup schemes $\rho: \Mr \otimes A \rightarrow \GLmn \otimes A$ is equivalent to defining a rational $\Mr \otimes A$-supermodule structure on the superspace $\Amn$, which is in turn equivalent to defining a right $A[\Mr]$-super\-comodule structure on $\Amn$. So let
\[
\Delta_\rho: \Amn \rightarrow \Amn \otimes_A A[\Mr] = \Amn \otimes_k k[\Mr]
\]
be the super\-comodule structure map corresponding to $\rho$. Then for $w \in \Amn$, we can write
\begin{equation} \label{eq:Deltarho}
\Delta_\rho(w) = \sum_{i=0}^{p^{r-1}-1} \sum_{j \geq 0} \left[\alpha_{ij}(w) \otimes  \theta^i \sigma_j + \beta_{ij}(w) \otimes \tau \theta^i \sigma_j\right]
\end{equation}
for some $k$-linear maps $\alpha_{ij},\beta_{ij} : \Amn \rightarrow \Amn$. Define $\wt{\beta}_{ij}: \Amn \rightarrow \Amn$ by $\wt{\beta}_{ij}(w) = (-1)^{\ol{w}}\beta_{ij}(w)$. Then it immediately follows that $\alpha_{ij} \in\Matmn(A)_{\zero}$ and $\wt{\beta}_{ij} \in \Matmn(A)_{\one}$. (If $A$ is a purely even $k$-super\-algebra, then $\beta_{ij}$ is already an element of $\Matmn(A)_{\one}$.)

Write $\Delta_{A[\Mr]}$ and $\ve$ for the coproduct and counit on $A[\Mr]$, respectively. Then the comodule axiom $(1 \otimes \ve) \circ \Delta_\rho = \id_{\Amn}$ implies that $\alpha_{0,0} = \id_{\Amn}$, while the axiom $(1 \otimes \Delta_{A[\Mr]}) \circ \Delta_\rho = (\Delta_\rho \otimes 1) \circ \Delta_\rho$ implies the following identities:
	\begin{align*}
	\alpha_{ij} \circ \alpha_{st} &= \binom{i+s}{i} \alpha_{i+s,j+t} \text{ if $i+s < p^{r-1}$}, & \alpha_{ij} \circ \alpha_{st} &=0 \text{ if $i+s \geq p^{r-1}$}, \\
	\beta_{ij} \circ \beta_{st} &= \binom{i+s}{i} \alpha_{i+s,j+t+p}, \text{ if $i+s < p^{r-1}$}, & \beta_{ij} \circ \beta_{st} &= 0 \text{ if $i+s \geq p^{r-1}$}, \\
	\alpha_{ij} \circ \beta_{st} &= \binom{i+s}{i} \beta_{i+s,j+t} \text{ if $i+s < p^{r-1}$}, & \alpha_{ij} \circ \beta_{st} &=0 \text{ if $i+s \geq p^{r-1}$}, \\
	\beta_{st} \circ \alpha_{ij} &= \binom{i+s}{i} \beta_{i+s,j+t} \text{ if $i+s < p^{r-1}$}, & \beta_{st} \circ \alpha_{ij} &=0 \text{ if $i+s \geq p^{r-1}$}.
	\end{align*}
In particular, the $\alpha_{ij}$ and $\beta_{ij}$ all commute pairwise, $(\alpha_{ij})^p = 0$ if $i \neq 0$, $\alpha_{0,j} = (\alpha_{0,1})^j$, and $(\beta_{0,0})^2 = \alpha_{0,p}$. Also, for each fixed $w \in \Amn$ the sum \eqref{eq:Deltarho} is finite, so $\alpha_{ij}(w) = 0$ for all but finitely many $j$. Then varying $w$ over an $A$-supermodule basis for $\Amn$, this implies that $\alpha_{ij} = 0$ for all but finitely many $j$, and hence that $(\alpha_{0,1})^N = \alpha_{0,N} = 0$ for some $N \geq 0$.

	
Now set $\alpha_i = \alpha_{p^i,0}$ for $0 \leq i \leq r-2$, set $\alpha_{r-1} = \alpha_{0,1}$, and set $\beta = -\wt{\beta}_{0,0}$. Then $\beta^2 = (\wt{\beta}_{0,0})^2 = -(\beta_{0,0})^2$, and the observations of the previous paragraph imply that the tuple $(\alpha_0,\ldots,\alpha_{r-1},\beta)$ is an element of $\bsV_r(\GLmn)(A)$ such that $\alpha_{r-1}$ is nilpotent. Conversely, let $(\alpha_0,\ldots,\alpha_{r-1},\beta)$ be such a tuple. Define $\beta_{0,0}: \Amn \rightarrow \Amn$ by $\beta_{0,0}(w) = -(-1)^{\ol{w}} \beta(w)$, i.e., so that $\beta = -\wt{\beta}_{0,0}$, and given an integer $0 \leq i < p^{r-1}$ with $p$-adic decomposition $i = i_0 + i_1 p + \cdots +i_{r-2} p^{r-2}$, set
\[
\alpha_{ij} = \frac{(\alpha_0)^{i_0}(\alpha_1)^{i_1} \cdots (\alpha_{r-2})^{i_{r-2}}}{(i_0)! (i_1)! \cdots (i_{r-2})!} (\alpha_{r-1})^j \quad \text{and} \quad \beta_{ij} = \alpha_{ij} \cdot \beta_{0,0}.
\]
Then with these definitions, \eqref{eq:Deltarho} defines an $A[\Mr]$-super\-comodule structure on $\Amn$. Thus, the elements of $\bfHom(\Mr,\GLmn)(A)$ correspond bijectively to the elements of $\bsV_r(\GLmn)(A)$ such that $\alpha_{r-1}$ is nilpotent. Finally, the action of an element $\phi \in A\Mr$ on an element $w \in \Amn$ is defined by $\phi \cdot w = (1 \otimes_A \phi)(\Delta_\rho(w))$. Then the description of the action of $A\Mr$ on $\Amn$ immediately follows because $u_i \in k[\Mr]^\#$ is defined to be the functional that is linearly dual to the basis element $\theta^{p^i} \in k[\Mr]$, and $v \in k[\Mr]^\#$ is defined to be linearly dual to $\tau \in k[\Mr]$.
\end{proof}

\begin{remark} \label{remark:exp}
Recall that the right action of $A$ on $\Matmn(A) = \Hom_A(\Amn,\Amn)$ is defined by $(\phi \cdot a)(w) = (-1)^{\ol{a} \cdot \ol{w}}\phi(w) \cdot a$. In particular, if $\tau \in B_{\one}$ for some commutative $A$-superalgebra $B \in \csalg_A$, and if $\beta = -\wt{\beta}_{0,0} \in \Matmn(A)_{\one}$ as in the proof, then (via the homomorphism $A \rightarrow B$ making $B$ into an $A$-superalgebra) one has $\beta \cdot \tau \in \Matmn(B)_{\zero}$. Note that while $\beta^2$ may not equal zero, $(\beta \cdot \tau)^2 = 0$ because $\tau^2 = 0$ (by the supercommutativity of $B$). For $w \in B^{m|n}$, one has
\[
(-\beta \cdot \tau)(w) = (-1)^{\ol{w}} (-\beta)(w) \cdot \tau = \beta_{0,0}(w) \cdot \tau.
\]

The preceding discussion and the proof of the proposition now imply the following explicit description for the homomorphism $\rho_{(\ualpha|\beta)} : \Mr \otimes_k A \rightarrow \GLmn \otimes_k A$ corresponding to an element $(\ualpha|\beta) \in \bsN_r(\GLmn)(A)$. Let $B \in \csalg_A$. For $\phi \in \Matmn(B)_{\zero}$, set
\[
\exp(\phi) = 1 + \phi + \frac{1}{2!} \phi^2 + \cdots + \frac{1}{(p-1)!}\phi^{p-1} \in \Matmn(B)_{\zero}.
\]
Then for $g = (\tau,\theta,\sigma_1,\sigma_2,\ldots,) \in \Mr(B) = (\Mr \otimes_k A)(B)$ (cf.\ the shorthand of \cite[\S3.1]{Drupieski:2017a} for denoting an element of $\Mr(B)$), one has
\begin{equation} \label{eq:rhoalphabeta} \textstyle
\rho_{(\ualpha|\beta)}(g) = \left( \prod_{i=0}^{r-1} \exp(\alpha_i \cdot \theta^{p^i}) \right) \cdot \exp(-\beta \cdot \tau) \cdot \left( \prod_{i \geq 1} \exp(\alpha_{r-1}^{p^i} \cdot \sigma_{p^i}) \right) \in \GLmn(B).
\end{equation}
This expression is well-defined, independent of the ordering of the factors, by the assumption that the entries of $(\ualpha|\beta)$ commute pairwise and $\alpha_{r-1}$ is nilpotent. If $\alpha_{r-1}^{p^s} = 0$, then \eqref{eq:rhoalphabeta} also describes the element of $\bfHom(\Mrs,\GLmn)(A)$ corresponding to $(\ualpha|\beta)$ as in \cite[Proposition 3.3.5]{Drupieski:2017a}.
\end{remark}

Recall from \cite[Proposition 3.3.5]{Drupieski:2017a} that $\bfHom(\Mrs,\GLmn)$ identifies with the closed subfunctor $\bsVrs(\GLmn)$ of $\bsVr(\GLmn)$, consisting of those tuples $(\alpha_0,\ldots,\alpha_{r-1},\beta) \in \bsVr(\GLmn)$ such that $\alpha_{r-1}^{p^s} = 0$. More generally, if $G$ is an algebraic $k$-super\-group scheme and if $G \hookrightarrow \GLmn$ is a closed embedding, then the proof of \cite[Theorem 3.3.7]{Drupieski:2017a} shows that $\bfHom(\Mrs,G)$ identifies with a closed subfunctor of $\bfHom(\Mrs,\GLmn)$. 

For each $s \geq 1$, the canonical quotient homomorphisms $\Mr \twoheadrightarrow \M_{r;s+1} \twoheadrightarrow \Mrs$ arising from the subalgebra inclusions $k[\Mrs] \hookrightarrow k[\M_{r;s+1}] \hookrightarrow k[\Mr]$ induce natural transformations
	\[
	\bfHom(\Mrs,G) \hookrightarrow \bfHom(\M_{r;s+1},G) \hookrightarrow \bfHom(\Mr,G).
	\]
In this way, the $\bfHom(\Mrs,G)$ for $s \geq 1$ form a directed system of subfunctors of $\bfHom(\Mr,G)$. The first part of the next proposition is then another interpretation of \cite[Proposition 3.3.5]{Drupieski:2017a}, while the second part follows from the observation \cite[Remark 3.1.3(4)]{Drupieski:2017a} that if $G$ is algebraic, then any $A$-super\-group scheme homomorphism $\rho: \Mr \otimes_k A \rightarrow G \otimes_k A$ necessarily factors for all sufficiently large $s$ through the canonical quotient map $\Mr \otimes_k A \twoheadrightarrow \Mrs \otimes_k A$.

\begin{lemma}
The functor $\bfHom(\Mr,\GLmn)$ is the union in $\bsVr(\GLmn)$ of the closed sub\-functors $\bfHom(\Mrs,\GLmn)$ for $s \geq 1$. In other words, as subfunctors of $\bsVr(\GLmn)$,
	\[
	\bfHom(\Mr,\GLmn) = \bigcup_{s \geq 1} \bfHom(\Mrs,\GLmn) = \bigcup_{s \geq 1} \bsVrs(\GLmn) = \bsNr(\GLmn)
	\]
More generally, let $G$ be an algebraic $k$-supergroup scheme, and fix a closed embedding $G \hookrightarrow \GLmn$ for some $m,n \in \N$. Then identifying $\bfHom(\Mr,G)$ and each $\bfHom(\Mrs,G)$ with a subfunctor of $\bfHom(\Mr,\GLmn)$ via the embedding, one has $\bfHom(\Mr,G) = \bigcup_{s \geq 1} \bfHom(\Mrs,G)$.
\end{lemma}

The functor $\bsNr(\GLmn)$ is not a closed subfunctor of $\bsVr(\GLmn)$ because there is no set of polynomial equations that simultaneously captures, for all coefficient rings $A \in \csalg_k$, the property that $\alpha_{r-1}$ is nilpotent. But by restricting attention to field coefficients we get:

\begin{lemma} \label{lem:HomMrGvariety}
Let $G$ be an algebraic $k$-supergroup scheme, let $G \hookrightarrow \GLmn$ be a closed embedding, and set $N = \max(m,n)$. Then for any field extension $K$ of $k$, the unions
	\begin{align*}
	\bfHom(\Mr,\GLmn)(K) &= \textstyle \bigcup_{s \geq 1} \bfHom(\Mrs,\GLmn)(K), \quad \text{and} \\
	\bfHom(\Mr,G)(K) &= \textstyle \bigcup_{s \geq 1} \bfHom(\Mrs,G)(K)
	\end{align*}
reach a stable value at $s = N$. In particular, if $k$ is algebraically closed, then $\bsNr(\GLmn)(k)$ is a Zariski closed subset of the affine algebraic variety $\bsVr(\GLmn)(k)$, and $\bfHom(\Mr,G)(k)$ identifies via the embedding $G \hookrightarrow \GLmn$ with a closed subvariety of $\bsNr(\GLmn)(k)$.
\end{lemma}

\begin{proof}
By \cite[Theorem 3.3.7]{Drupieski:2017a}, the functor $\bsVr(\GLmn)$ admits the structure of an affine superscheme of finite type over $k$. Let $A = k[\bsVr(\GLmn)]$ be the coordinate superalgebra of this affine superscheme, and let $A_{red}$ be the largest (purely even) reduced quotient of $A$.\footnote{Since $A$ is a commutative superalgebra, its odd elements are automatically nilpotent, and hence the largest reduced quotient of $A$ will automatically be a purely even $k$-algebra.} Then $A_{red}$ is a finitely-generated commutative (in the ordinary sense) $k$-algebra, and because the field $k$ is a purely even reduced $k$-algebra,
	\[
	\bsVr(\GLmn)(k) = \Hom_{\salg}(A,k) = \Hom_{\alg}(A_{red},k).
	\]
In particular, if $k$ is algebraically closed then $\bsVr(\GLmn)(k)$ admits the structure of an affine algebraic variety with coordinate algebra $A_{red}$.

Next consider the defining property of $\bsNr(\GLmn)(K)$ that $\alpha_{r-1}$ is nilpotent. Since the matrix coefficients lie in the field $K$, it follows (e.g., from considering Jordan canonical forms, perhaps over some algebraically closed extension field) that $\alpha_{r-1}$ is nilpotent if and only if $(\alpha_{r-1})^N = 0$. This implies that $\bsVrs(\GLmn)(K) = \bsV_{r;N}(\GLmn)(K)$ for all $s \geq N$, which in turn implies that the unions $\bsNr(\GLmn)(K) = \bigcup_{s \geq 1} \bfHom(\Mrs,\GLmn)(K)$ and $ \bigcup_{s \geq 1} \bfHom(\Mrs,G)(K)$ reach stable values at $s=N$. Finally, by the proof of \cite[Theorem 3.3.7]{Drupieski:2017a}, $\bfHom(\M_{r;N},\GLmn)$ identifies with a closed subsuperscheme of $\bsVr(\GLmn)$, and $\bfHom(\M_{r;N},G)$ identifies via the embedding $G \hookrightarrow \GLmn$ with a closed subsuperscheme of $\bfHom(\M_{r;N},\GLmn)$. Then taking $k$-points, it follows as in the first paragraph that if $k$ is algebraically closed, then $\bsNr(\GLmn)(k) = \bfHom(\M_{r;N},\GLmn)(k)$ identifies with a Zariski closed subset of the affine variety $\bsV(\GLmn)(k)$, and $\bfHom(\M_{r;N},G)(k)$ identifies with a closed subvariety of $\bsNr(\GLmn)(k)$.
\end{proof}

\begin{definition} \label{definition:kNrG}
Let $G$ be an algebraic $k$-supergroup scheme. By Lemma \ref{lem:HomMrGvariety}, there exists a minimal integer $N \geq 1$ such that for any field extension $K$ of $k$ and any integer $N' \geq N$, the canonical quotient map $\Mr \twoheadrightarrow \M_{r;N'}$ induces an identification $\bfHom(\M_{r;N'},G)(K) = \bfHom(\Mr,G)(K)$. (For $G = \GLmn$, one has $N = \max(m,n)$.) Define the $k$-algebra $k[\NrG]$ by
	\[
	k[\NrG] = k[\bfHom(\M_{r;N},G)]_{red},
	\]
the largest (purely even) reduced quotient of the coordinate superalgebra $k[\bfHom(\M_{r;N},G)]$. If $k$ is algebraically closed, then $k[\NrG]$ is the coordinate algebra of the affine algebraic variety
	\[
	\NrG := \bfHom(\Mr,G)(k).
	\]
\end{definition}

\begin{remark} \label{remark:kNrG}
Let $G$ be an algebraic $k$-supergroup scheme, and let $N$ be as in Definition \ref{definition:kNrG}.
\begin{enumerate}
\item The affine $k$-superscheme structure on $\bfHom(\M_{r;N},G)$ is defined in \cite{Drupieski:2017a} in terms of a choice of closed embedding $G \hookrightarrow \GLmn$. However, Yoneda's Lemma then implies that the coordinate super\-algebra $k[\bfHom(\M_{r;N},G)]$ is independent, up to isomorphism, of the choice of closed embedding. Consequently, $k[\NrG]$ is also uniquely determined up to isomorphism.

\item Let $N' \geq N$. Then by the definition of $N$, the canonical quotient maps $\Mr \twoheadrightarrow \M_{r;N'} \twoheadrightarrow \M_{r;N}$ induce identifications
	\begin{equation} \label{eq:HomKequal}
	\bfHom(\M_{r;N},G)(k) = \bfHom(\M_{r;N'},G)(k) = \bfHom(\Mr,G)(k).
	\end{equation}
The quotient map $\pi: \M_{r;N'} \twoheadrightarrow \M_{r;N}$ also induces a $k$-algebra homomorphism
	\[
	\pi^*: k[\bfHom(\M_{r;N'},G)]_{red} \rightarrow k[\bfHom(\M_{r;N},G)]_{red} = k[\NrG].
	\]
If $k$ is algebraically closed, then \eqref{eq:HomKequal} implies that $\pi^*$ induces the identity map between maximal ideal spectra, and hence implies (by the anti-equivalence of categories between affine algebraic varieties over $k$, and finitely-generated reduced commutative $k$-algebras) that $\pi^*$ is an isomorphism.

\item \label{item:basechange} Let $K/k$ be a field extension. Base change to $K$ defines a $k$-algebra homomorphism
	\[
	k[\bfHom(\M_{r;N},G)] \rightarrow k[\bfHom(\M_{r;N},G)] \otimes_k K = K[\bfHom(\M_{r;N} \otimes_k K,G \otimes_k K)].
	\]
Then passing to the reduced quotient rings, one gets a $k$-algebra homomorphism
	\[
	k[\NrG] \rightarrow K[\Nr(G \otimes_k K)].
	\]
\end{enumerate}
\end{remark}

\begin{lemma} \label{lemma:kNr}
Suppose $k$ is algebraically closed, and let $\bsmu,\bsa_0,\ldots,\bsa_{r-1},\bsb_s$ be indeterminates over $k$. Then
	\begin{align*}
	k[\Nr(\Mrone)] &\cong k[\bsmu,\bsa_0,\ldots,\bsa_{r-1}], \\
	k[\Nr(\Mrs)] &\cong k[\bsmu,\bsa_0,\ldots,\bsa_{r-1},\bsb_s]/\subgrp{\bsmu^2 - \bsa_0^{p^r}} & \text{if $s \geq 2$,} \\
	k[\Nr(\Mrseta)] &\cong k[\bsmu,\bsa_0,\ldots,\bsa_{r-1}]/\subgrp{\bsmu^2 - \bsa_0^{p^r}} & \text{if $r \geq 2$ and $0 \neq \eta \in k$,} \\
	k[\Nr(\Gar)] &\cong k[\bsa_0,\ldots,\bsa_{r-1}], & \text{and} \\
	k[\Nr(\Gam)] &\cong k[\bsmu].
	\end{align*}
\end{lemma}

\begin{proof}
These characterizations are immediate consequences of the calculation in Proposition \ref{prop:identifyhom} of the variety $\NrG = \bfHom(\Mr,G)(k)$.
\end{proof}

\begin{lemma} \label{lemma:NrGnatural}
Suppose $k$ is algebraically closed and let $\phi: G \rightarrow G'$ be a homomorphism of algebraic $k$-supergroup schemes. Then composition with $\phi$ defines a morphism of affine varieties
	\[
	\phi_*: \NrG \rightarrow \Nr(G'), \quad \nu \mapsto \phi \circ \nu,
	\]
and hence also a homomorphism of $k$-algebras $\phi^*: k[\Nr(G')] \rightarrow k[\NrG]$.
\end{lemma}

\begin{proof}
Choose a positive integer $N$ large enough so that $\NrG = \bfHom(\M_{r;N},G)(k)$ and $\Nr(G') = \bfHom(\M_{r;N},G')(k)$. By \cite[Theorem 3.3.7]{Drupieski:2017a}, composition with $\phi$ defines a morphism of affine $k$-super\-schemes $\bfHom(\M_{r;N},G) \rightarrow \bfHom(\M_{r;N},G')$. Then it follows that the induced map between the sets of $k$-points is a morphism of affine varieties.
\end{proof}

\subsection{Universal homomorphisms} \label{subsection:universalhomomorphism}

Let $G$ be an algebraic $k$-supergroup scheme. Recall from \cite[Definition 3.3.9]{Drupieski:2017a} the universal super\-group homomorphism
	\[
	\rho : \Mrs \otimes_k k[\bfHom(\Mrs,G)] \rightarrow G \otimes_k k[\bfHom(\Mrs,G)].
	\]
This homomorphism is universal in the sense that if $B \in \csalg_k$ and if $\rho' \in \bfHom(\Mrs,G)(B)$, then there exists a unique $k$-superalgebra homomorphism $\phi: k[\bfHom(\Mrs,G)] \rightarrow B$ such that $\rho' = \rho \otimes_\phi B$, i.e., such that $\rho' : \Mrs \otimes_k B \rightarrow G \otimes_k B$ is obtained from $\rho$ via base change along $\phi$.

\begin{definition}[Universal homomorphism from $\Mr$ to $G$]
Let $G$ be an algebraic $k$-supergroup scheme, and let $N \geq 1$ be as in Definition \ref{definition:kNrG}. Define the \emph{universal supergroup homomorphism from $\Mr$ to $G$} to be the homomorphism of $k[\NrG]$-supergroup schemes
	\[
	u_G: \Mr \otimes_k k[\NrG] \rightarrow G \otimes_k k[\NrG]
	\]
that is obtained from the universal homomorphism
	\[
	\rho : \M_{r;N} \otimes_k k[\bfHom(\M_{r;N},G)] \rightarrow G \otimes_k k[\bfHom(\M_{r;N},G)]
	\]
via base change along the canonical quotient map
	\[
	k[\bfHom(\M_{r;N},G)] \twoheadrightarrow k[\bfHom(\M_{r;N},G)]_{red} = k[\NrG].
	\]
Then $u_G$ is universal in the sense that if $K/k$ is a field extension and if $\rho' \in \bfHom(\Mr,G)(K)$, then there exists a unique $k$-algebra homomorphism $\phi: k[\NrG] \rightarrow K$ such that $\rho' = u_G \otimes_\phi K$, i.e., such that $\rho': \Mr \otimes_k K \rightarrow G \otimes_k K$ is obtained from $u_G$ via base change along $\phi$.
\end{definition}

Let $G$ be an algebraic $k$-supergroup scheme. As in \cite[\S6.1]{Drupieski:2017a}, set
	\begin{equation} \label{eq:HGk}
	H(G,k) := \opH^{\ev}(G,k)_{\zero} \oplus \opH^{\odd}(G,k)_{\one}.
	\end{equation}
Then $H(G,k)$ inherits from $\Hbul(G,k)$ the structure of a $\Z$-graded commutative (in the ordinary sense) $k$-algebra. As in \cite[\S6.2]{Drupieski:2017a}, we use the homomorphism $u_G: \Mr \otimes_k k[\NrG] \rightarrow G \otimes_k k[\NrG]$ to define a $k$-algebra homomorphism
	\[
	\psi_r: H(G,k) \rightarrow k[\NrG].
	\]
Recall from \cite[Proposition 3.2.1(2)]{Drupieski:2017a} the calculation of the cohomology ring $\Hbul(\Mr,k)$:
	\[
	\Hbul(\Mr,k) \cong k[x_1,\ldots,x_r,y]/\subgrp{x_r-y^2} \gotimes \Lambda(\lambda_1,\ldots,\lambda_r).
	\]
Then the map $\Hbul(\Mr,k) \rightarrow k$ that sends $x_r$ and $y$ each to $1$ but that sends the other algebra generators to $0$ is a $k$-algebra homomorphism (though not a superalgebra homomorphism, since $y$ is of odd superdegree). Extending scalars to $k[\NrG]$, one gets a $k$-algebra homomorphism
	\[
	\ve: \Hbul(\Mr \otimes_k k[\NrG],k[\NrG]) = \Hbul(\Mr,k) \otimes_k k[\NrG] \rightarrow k \otimes_k k[\NrG] = k[\NrG].
	\]
Now define $\psi_r: H(G,k) \rightarrow k[\NrG]$ to be the composite algebra homomorphism
	\begin{multline} \label{eq:psir}
	\psi_r: H(G,k) \stackrel{\iota}{\longrightarrow} \Hbul(G,k) \otimes_k k[\NrG] = \Hbul(G \otimes_k k[\NrG],k[\NrG]) \\
	\stackrel{u_G^*}{\longrightarrow} \Hbul(\Mr \otimes_k k[\NrG],k[\NrG]) \stackrel{\ve}{\longrightarrow} k[\NrG],
	\end{multline}
where $\iota: H(G,k) \rightarrow \Hbul(G,k) \otimes_k k[\NrG]$ is the base change map $z \mapsto z \otimes 1$. Thus, in the special case $r=1$ and for $z \in H(G,k)$ homogeneous of degree $n$, one gets
	\begin{equation} \label{eq:uG*r=1}
	(u_G^* \circ \iota)(z) = \psi(z) \cdot y^n \in H(\Mone,k) \otimes_k k[\NoneG] \cong k[y] \otimes_k k[\NoneG].
	\end{equation}
The next proposition is an analogue of \cite[Theorem 1.14]{Suslin:1997a}, and is related to \cite[Proposition 6.2.2]{Drupieski:2017a}.

\begin{proposition} \label{prop:psirgradedmap}
Let $G$ be an algebraic $k$-supergroup scheme. Then $k[\NrG]$ admits the structure of a $\Z[\frac{p^r}{2}]$-graded connected $k$-algebra, and the homomorphism
	\[
	\psi_r: H(G,k) \rightarrow k[\NrG]
	\]
is then a homomorphism of graded $k$-algebras that multiplies degrees by $\frac{p^r}{2}$. Moreover, $\psi_r$ is natural with respect to homomorphisms $\phi: G \rightarrow G'$ of algebraic $k$-supergroup schemes.
\end{proposition}

\begin{proof}
Let $N$ be as in Definition \ref{definition:kNrG}. Recall from \cite[Definition 3.3.8]{Drupieski:2017a} that $k[V_{r;N}(G)]$ is the largest purely even quotient of $k[\bfHom(\M_{r;N},G)]$. Then $k[\NrG] = k[V_{r;N}(G)]_{red}$. By \cite[Corollary 3.4.3]{Drupieski:2017a}, $k[V_{r;N}(G)]$ admits the structure of a $\Z[\frac{p^r}{2}]$-graded connected $k$-algebra. Then the nilpotent elements in $k[V_{r;N}(G)]$ form a homogeneous ideal, and hence $k[\NrG]$ inherits the structure of a $\Z[\frac{p^r}{2}]$-graded connected $k$-algebra. Next, recall from \cite[\S6.2]{Drupieski:2017a} that the homomorphism
	\[
	\psi_{r;N}: H(G,k) \rightarrow k[V_{r;N}(G)]
	\]
is defined as a composite exactly as in \eqref{eq:psir}, except that $k[\NrG]$ is replaced by $k[V_{r;N}(G)]$, and the homomorphism $u_G: \Mr \otimes_k k[\NrG] \rightarrow G \otimes_k k[\NrG]$ is replaced by the universal purely even supergroup homomorphism $u_{r;N}: \M_{r;N} \otimes_k k[V_{r;N}(G)] \rightarrow G \otimes_k k[V_{r;N}(G)]$ of \cite[Definition 3.3.9(2)]{Drupieski:2017a}. Since $u_G$ can be obtained from $u_{r;N}$ via base change along the canonical quotient map $k[V_{r;N}(G)] \twoheadrightarrow k[V_{r;N}(G)]_{red} = k[\NrG]$, it follows that there is a commutative diagram
\begin{equation} \label{eq:psirdiagram}
\xymatrix@C=1em{
H(G,k) \ar@{->}[r] \ar@{=}[d] & \Hbul(G \otimes_k k[V_{r;N}(G)],k[V_{r;N}(G)]) \ar@{->}[r] \ar@{->}[d] & \Hbul(\Mr \otimes_k k[V_{r;N}(G)],k[V_{r;N}(G)]) \ar@{->}[r] \ar@{->}[d] & k[V_{r;N}(G)] \ar@{->>}[d] \\
H(G,k) \ar@{->}[r] & \Hbul(G \otimes_k k[\NrG],k[\NrG]) \ar@{->}[r] & \Hbul(\Mr \otimes_k k[\NrG],k[\NrG]) \ar@{->}[r] & k[\NrG]
}
\end{equation}
in which the rows are the composites defining $\psi_{r;N}$ and $\psi_r$, respectively, and the vertical arrows are induced by the quotient map $k[V_{r;N}(G)] \twoheadrightarrow k[\NrG]$. By \cite[Proposition 6.6.2]{Drupieski:2017a}, $\psi_{r;N}$ is a homomorphism of graded $k$-algebras that multiplies degress by $\frac{p^r}{2}$. Then by the commutativity of the diagram it follows that $\psi_r$ is as well. Finally, the last assertion concerning naturality in $G$ follows from the same line of reasoning as the proof of \cite[Lemma 6.2.1]{Drupieski:2017a}.
\end{proof}

\section{Cohomology and support varieties} \label{section:cohomologyandsupport}

\subsection{Induced maps in cohomology} \label{subsection:inducedmaps}

In this section we calculate the maps in cohomology induced by the homomorphisms described in Proposition \ref{prop:identifyhom}. We identify the relevant cohomology rings as in \cite[Proposition 3.2.1]{Drupieski:2017a}, using also the fact that base change induces identifications of the form $\Hbul(G \otimes_k A,A) = \Hbul(G,k) \otimes_k A$.

\begin{lemma} \label{lemma:inducedmapincohomology}
Let $A = \Azero \in \calg_k$ be a purely even commutative $k$-algebra. In parts (2) and (3), assume that $A$ is reduced.
	\begin{enumerate}
	\item \label{item:rho*Mrone} Let $(\mu,a_0,\ldots,a_{r-1}) \in A^{r+1}$, and let $\rho: \Mrone \otimes_k A \rightarrow \Mrone \otimes_k A$ be the corresponding homo\-morphism as in Proposition \ref{prop:identifyhom}(\ref{item:HomMrone}). Then the induced map in cohomology $\rho^*: \Hbul(\Mrone \otimes_k A,A) \rightarrow \Hbul(\Mrone \otimes_k A,A)$ satisfies
		\begin{align*}
		\rho^*(y) &= \mu \cdot y, \\
		\rho^*(\lambda_i) &= a_0^{p^{i-1}} \lambda_i + a_1^{p^{i-1}} \lambda_{i+1} + \cdots + a_{r-i}^{p^{i-1}} \lambda_r & \text{for $1 \leq i \leq r$, and} \\
		\rho^*(x_i) &= a_0^{p^i} x_i + a_1^{p^i} x_{i+1} + \cdots + a_{r-i}^{p^i} x_r & \text{for $1 \leq i \leq r$.}
		\end{align*}
	
	\item \label{item:rho*Mrs} Suppose $s \geq 2$. Let $(\mu,a_0,\ldots,a_{r-1},b_s) \in A^{r+2}$ such that $\mu^2 = a_0^{p^r}$, and let $\rho: \Mrs \otimes_k A \rightarrow \Mrs \otimes_k A$ be the corresponding homomorphism as in Proposition \ref{prop:identifyhom}(\ref{item:HomMrs}). Then the induced map in cohomology $\rho^*: \Hbul(\Mrs \otimes_k A,A) \rightarrow \Hbul(\Mrs \otimes_k A,A)$ satisfies
		\begin{align*}
		\rho^*(y) &= \mu \cdot y, \\
		\rho^*(\lambda_i) &= a_0^{p^{i-1}} \lambda_i + a_1^{p^{i-1}} \lambda_{i+1} + \cdots + a_{r-i}^{p^{i-1}} \lambda_r & \text{for $1 \leq i \leq r$,} \\
		\rho^*(x_i) &= a_0^{p^i} x_i + a_1^{p^i} x_{i+1} + \cdots + a_{r-i}^{p^i} x_r & \text{for $1 \leq i \leq r$, and} \\
		\rho^*(w_s) &= a_0^{p^{r+s-1}} w_s + (b_s)^p x_r.
		\end{align*}
	
	\item \label{item:rho*Mrseta} Suppose $r \geq 2$, and let $0 \neq \eta \in k$. Let $(\mu,a_0,\ldots,a_{r-1}) \in A^{r+1}$ such that $\mu^2 = a_0^{p^r}$, and let $\rho: \M_{r;s+1} \otimes_k A \rightarrow \Mrseta \otimes_k A$ be the corresponding homomorphism as in Proposition \ref{prop:identifyhom}(\ref{item:HomMrseta}). Then identifying $\Hbul(\Mrseta,k)$ as an algebra with
		\[
		\Hbul(\M_{r-1;s+1},k) = k[x_1,\ldots,x_{r-1},y,w_{s+1}] /\subgrp{x_{r-1}-y^2} \gotimes \Lambda(\lambda_1,\ldots,\lambda_{r-1})
		\]
	as in \cite[Proposition 3.2.1(5)]{Drupieski:2017a}, the induced map in cohomology $\rho^*: \Hbul(\Mrseta \otimes_k A,A) \rightarrow \Hbul(\M_{r;s+1} \otimes_k A,A)$ satisfies
		\begin{align*}
		\rho^*(y) &= \mu \cdot y, \\
		\rho^*(\lambda_i) &= a_0^{p^i} \lambda_{i+1} + a_1^{p^i} \lambda_{i+2} + \cdots + a_{r-i-1}^{p^i} \lambda_r & \text{for $1 \leq i \leq r-1$,} \\
		\rho^*(x_i) &= a_0^{p^{i+1}} x_{i+1} + a_1^{p^{i+1}} x_{i+2} + \cdots + a_{r-i-1}^{p^{i+1}} x_r & \text{for $1 \leq i \leq r-1$, and} \\
		\rho^*(w_{s+1}) &= (-\eta^{-1})^p \cdot (a_0^p x_1 + a_1^p x_2 + \cdots + a_{r-1}^p x_r) + a_0^{p^{r+s}} w_{s+1}.
		\end{align*}

	\item \label{item:rho*Gar} Let $(a_0,\ldots,a_{r-1}) \in A^r$, and let $\rho: \Gar \otimes_k A \rightarrow \Mrone \otimes_k A$ be the corresponding homo\-mor\-phism as in Proposition \ref{prop:identifyhom}(\ref{item:HomGar}). Then the induced map in cohomology $\rho^*: \Hbul(\Gar \otimes_k A,A) \rightarrow \Hbul(\Mrone \otimes_k A,A)$ satisfies
		\begin{align*}
		\rho^*(\lambda_i) &= a_0^{p^{i-1}} \lambda_i + a_1^{p^{i-1}} \lambda_{i+1} + \cdots + a_{r-i}^{p^{i-1}} \lambda_r & \text{for $1 \leq i \leq r$, and} \\
		\rho^*(x_i) &= a_0^{p^i} x_i + a_1^{p^i} x_{i+1} + \cdots + a_{r-i}^{p^i} x_r & \text{for $1 \leq i \leq r$.}
		\end{align*}

	\item \label{item:rho*Gaminus} Let $\mu \in A$, and let $\rho: \Gam \otimes_k A \rightarrow \Mrone \otimes_k A$ be the corresponding homo\-morphism as in Proposition \ref{prop:identifyhom}(\ref{item:HomGaminus}). Write $\Hbul(\Gam,k) = k[y]$ as in \cite[Proposition 3.2.1]{Drupieski:2017a}. Then the induced map in cohomology $\rho^*: \Hbul(\Gam \otimes_k A,A) \rightarrow \Hbul(\Mrone \otimes_k A,A)$ satisfies $\rho^*(y) = \mu \cdot y$.
	\end{enumerate}
\end{lemma}

\begin{proof}
We will prove in detail statements \eqref{item:rho*Mrs} and \eqref{item:rho*Mrseta}, but will leave the proofs of the other (easier) statements to the reader. To begin, suppose $s \geq 2$, let $(\mu,a_0,\ldots,a_{r-1},b_s) \in A^{r+2}$ such that $\mu^2 = a_0^{p^r}$, and let $\rho: \Mrs \otimes_k A \rightarrow \Mrs \otimes_k A$ be the corresponding homomorphism as in Proposition \ref{prop:identifyhom}(\ref{item:HomMrs}). To describe the map in cohomology induced by $\rho$, we'll consider how the comorphism $\rho^*: A[\Mrs] \rightarrow A[\Mrs]$ acts on the explicit cochain representatives described in \cite[Proposition 3.2.1]{Drupieski:2017a} for the generators of the cohomology ring $\Hbul(\Mrs \otimes_k A,A) = \Hbul(\Mrs,k) \otimes_k A$. First, since $y = [\tau]$ and $\rho^*(\tau) = \mu \cdot \tau$, then $\rho^*(y) = \mu \cdot y$. Next,
	\[ \textstyle
	\rho^*(\theta^{p^{i-1}}) = \rho^*(\theta)^{p^{i-1}} = \left( \sum_{\ell=0}^{r-1} a_\ell \cdot \theta^{p^\ell} \right)^{p^{i-1}} = a_0^{p^{i-1}} \theta_i + a_1^{p^{i-1}} \theta_{i+1} + \cdots + a_{r-i}^{p^{i-1}} \theta_r.
	\]
Since $\lambda_i$ is the cohomology class of $\theta^{p^{i-1}}$, this implies the description of $\rho^*(\lambda_i)$ in \eqref{item:rho*Mrs}. Now recall from \cite[Remark 3.4.4]{Drupieski:2013b} the semilinear function $\ol{\beta}: \opH^1(\Mrs,k)_{\zero} \rightarrow \opH^2(\Mrs,k)_{\zero}$. Given a cocycle representative $f \in k[\Mrs]_{\zero}$ for an element of $\opH^1(\Mrs,k)_{\zero}$, the action of $\ol{\beta}$ on the cohomology class $[f]$ is defined by $\ol{\beta}([f]) = [\beta(f)]$, where
	\begin{equation} \label{eq:beta} \textstyle
	\beta(f) = \sum_{\ell=1}^{p-1} \cbinom{p}{\ell} f^\ell \otimes f^{p-\ell},
	\end{equation}
and $\cbinom{p}{\ell}$ denotes the integer $\frac{1}{p} \cdot \binom{p}{\ell}$. The map $\ol{\beta}$ is evidently natural with respect to Hopf superalgebra homomorphisms (equivalently, homomorphisms of supergroup schemes). Since $x_i = \ol{\beta}(\lambda_i)$ by the calculations in \cite[Proposition 3.2.1]{Drupieski:2017a} (cf.\ also the proof of \cite[Proposition 3.4.2]{Drupieski:2013b}), and since $\ol{\beta}$ is semilinear, the description of $\rho^*(x_i)$ in \eqref{item:rho*Mrs} follows from the description of $\rho^*(\lambda_i)$. Finally, recall from \cite[Remark 3.1.3]{Drupieski:2017a} that the superalgebra grading on $k[\Mrs]$ lifts to a $\Z$-grading. This $\Z$-grading endows the cohomology ring $\Hbul(\Mrs,k)$ with an internal $\Z$-grading in addition to the (external) cohomological grading. Then $\opH^2(\Mrs,k)_{\zero}$ is spanned by
	\begin{align*}
	x_i &\quad \text{for $1 \leq i \leq r$,} & \text{of internal degree $2p^i$,} \\
	\lambda_i \lambda_j & \quad \text{for $1 \leq i < j \leq r$,} & \text{of internal degree $2p^{i-1}+2p^{j-1} < 2p^r$, and} \\
	w_s & & \text{of internal degree $2p^{r+s-1}$.}
	\end{align*}
In particular, $\opH^2(\Mrs,k)_{\zero}$ has no nonzero elements of internal degrees strictly between $2p^r$ and $2p^{r+s-1}$. Now $\rho^*$ maps $w_s$ to the cohomology class of
	\begin{multline} \label{eq:rho*ws}
	-\left[ \sum_{j=1}^{p^s-1} \rho^*(\sigma_j) \otimes \rho^*(\sigma_{p^s-j}) + \sum_{i+j+p=p^s} \rho^*(\sigma_i \tau) \otimes \rho^*(\sigma_j \tau) \right] \\
	= -\left[ \sum_{j=1}^{p^s-1} \rho^*(\sigma_j) \otimes \rho^*(\sigma_{p^s-j}) + \sum_{i+j+p=p^s} \mu^2 \cdot \rho^*(\sigma_i) \tau \otimes \rho^*(\sigma_j) \tau \right].
	\end{multline}
If $0 \leq i < p^{s-1}$ and $0 \leq j < p$, then $\rho^*(\sigma_{i+jp^{s-1}})$ is given by formula \eqref{eq:phisigmaijpn} with $n=s-1$ and $c_{s-1} = b_s$. Using this observation, one can check that the component of greatest internal $\Z$-degree in \eqref{eq:rho*ws} is
	\[
	- a_0^{p^{r+s-1}} \cdot \left[ \sum_{j=1}^{p^s-1} \sigma_j \otimes \sigma_{p^s-j} + \sum_{i+j+p=p^s} \sigma_i \tau \otimes \sigma_j \tau \right],
	\]
which is of internal degree $2p^{r+s-1}$ and is a cochain representative for $a_0^{p^{r+s-1}} w_s$, and the component of smallest internal $\Z$-degree in \eqref{eq:rho*ws} is
	\[
	- (b_s)^p \cdot \left[ \sum_{i=1}^{p-1} \sigma_i \otimes \sigma_{p-i} \right],
	\]
which is of internal degree $2p^r$ and is a cochain representative for $(b_s)^p x_r$. Since $\opH^2(\Mrs,k)_{\zero}$ has no nonzero elements of internal degrees strictly between $2p^r$ and $2p^{r+s-1}$, the other homogeneous components of \eqref{eq:rho*ws} must be cohomologous to zero. Then it follows that $\rho^*(w_s) = a_0^{p^{r+s-1}} w_s + (b_s)^p x_r$, as desired. This completes the proof of \eqref{item:rho*Mrs}.

Now suppose $r \geq 2$ and $0 \neq \eta \in k$. The identification $\Hbul(\M_{r-1;s+1},k) \cong \Hbul(\Mrseta,k)$ arises from a coalgebra isomorphism $\pi^*: k[\M_{r-1;s+1}] \stackrel{\sim}{\rightarrow} k[\Mrseta]$, which in turn arises via duality from a group algebra isomorphism $\pi: k\Mrseta \stackrel{\sim}{\rightarrow} k\M_{r-1;s+1}$. The isomorphism $\pi$ is defined by $\pi(v) = v$, $\pi(u_i) = u_{i-1}$ for $1 \leq i \leq r$, and $\pi(u_0) = -\eta^{-1} \cdot u_{r-1}^s$; see \cite[Remark 3.1.8(4)]{Drupieski:2017a}. Then given integers $0 \leq i < p^{r-2}$, $0 \leq j < p^s$, and $0 \leq i_0 < p$, one can check via duality that the coalgebra isomorphism $\pi^*: k[\M_{r-1;s+1}] \rightarrow k[\Mrseta]$ satisfies
	\[
	\pi^*( \theta^i \sigma_{j+i_0 p^s}) = \frac{1}{(i_0)!} (-\eta^{-1})^{i_0} \cdot  \theta^{i_0+pi} \sigma_j \quad \text{and} \quad	\pi^*( \theta^i \sigma_{j+i_0 p^s}\tau) = \frac{1}{(i_0)!} (-\eta^{-1})^{i_0} \cdot  \theta^{i_0+pi} \sigma_j \tau.
	\]
Here as in \cite[Lemma 3.1.7]{Drupieski:2017a} we make the superalgebra identification $k[\Mrseta] \cong k[\Mrs]$.

Now let $(\mu,a_0,\ldots,a_{r-1}) \in A^{r+1}$ such that $\mu^2 = a_0^{p^r}$, and let $\rho: \M_{r;s+1} \otimes_k A \rightarrow \Mrseta \otimes_k A$ be the corresponding homomorphism as in Proposition \ref{prop:identifyhom}(\ref{item:HomMrseta}). Then to establish statement \eqref{item:rho*Mrseta} of the lemma we must describe the map in cohomology induced by the composite coalgebra homomorphism $\rho^* \circ \pi^* : A[\M_{r-1;s+1}] \rightarrow A[\Mrseta] \rightarrow A[\M_{r;s+1}]$. First, since $\pi^*(\tau) = \tau$, one gets
	\[
	(\rho^* \circ \pi^*)(y) = (\rho^* \circ \pi^*)([\tau]) = \rho^*([\tau]) = [\rho^*(\tau)] = [ \mu \cdot \tau ] = \mu \cdot y.
	\]
Next for $1 \leq i \leq r-2$ one has
	\begin{align*}
	(\rho^* \circ \pi^*)(\lambda_i) = (\rho^* \circ \pi^*)([\theta^{p^{i-1}}]) &= [ \rho^*(\theta^{p^i}) ] \\
	&= [ a_0^{p^i} \theta^{p^i} + a_1^{p^i} \theta^{p^{i+1}} + \cdots + a_{r-i-1}^{p^i} \theta^{p^{r-1}}] \\
	&= a_0^{p^i} \lambda_{i+1} + a_1^{p^i} \lambda_{i+2} + \cdots + a_{r-i-1}^{p^i} \lambda_r,
	\end{align*}
while for $i=r-1$ one has, because $\theta^{p^{r-2}} = \sigma_1$ in $k[\M_{r-1;s+1}]$ and $\theta^{p^{r-1}} = \sigma_1$ in $k[\M_{r;s+1}]$,
	\[
	(\rho^* \circ \pi^*)(\lambda_{r-1}) = (\rho^* \circ \pi^*)([\sigma_1]) = [ \rho^*(\sigma_1) ] = [ a_0^{p^{r-1}} \sigma_1] = [ a_0^{p^{r-1}} \theta^{p^{r-1}}] = a_0^{p^{r-1}} \lambda_r.
	\]
Now for $1 \leq i \leq r-2$ one gets
	\begin{align*}
	(\rho^* \circ \pi^*)(x_i) = (\rho^* \circ \pi^*)([\beta(\theta^{p^{i-1}})]) = \rho^*([\beta(\theta^{p^i})]) &= \ol{\beta}([\rho^*(\theta^{p^i})]) \\
	&= \ol{\beta}(a_0^{p^i} \lambda_{i+1} + a_1^{p^i} \lambda_{i+2} + \cdots + a_{r-i-1}^{p^i} \lambda_r) \\
	&= a_0^{p^{i+1}} x_{i+1} + a_1^{p^{i+1}} x_{i+2} + \cdots + a_{r-i-1}^{p^{i+1}} x_r,
	\end{align*}
while for $i=r-1$ one gets
	\[
	(\rho^* \circ \pi^*)(x_{r-1}) = (\rho^* \circ \pi^*)([\beta(\sigma_1)]) = \rho^*([\beta(\sigma_1)]) = \ol{\beta}([\rho^*(\sigma_1)]) = \ol{\beta}(a_0^{p^{r-1}} \lambda_r) = a_0^{p^r} x_r.
	\]
Finally,
	\begin{multline*}
	(\rho^* \circ \pi^*)(w_{s+1}) = (\rho^* \circ \pi^*) \left( -\left[ \sum_{j=1}^{p^{s+1}-1} \sigma_j \otimes \sigma_{p^{s+1}-j} + \sum_{i+j+p=p^{s+1}} \sigma_i \tau \otimes \sigma_j \tau \right] \right) \\
	= (\rho^* \circ \pi^*) \left( -\left[ \sum_{i=1}^{p-1} \sigma_{ip^s} \otimes \sigma_{(p-i)p^s} + \sum_{\substack{i_0+j_0=p-1 \\ i+j=p^s,\, i,j \geq 1}} \sigma_{i+i_0 p^s} \otimes \sigma_{j+j_0p^s} + \sum_{i+j+p=p^{s+1}} \sigma_i \tau \otimes \sigma_j \tau \right] \right)
	\end{multline*}
The first summation inside the square brackets can be rewritten as
	\[
	 \sum_{i=1}^{p-1} \sigma_{ip^s} \otimes \sigma_{(p-i)p^s} = -\sum_{i=1}^{p-1} \cbinom{p}{i} (\sigma_{p^s})^i \otimes (\sigma_{p^s})^{p-i} = -\beta(\sigma_{p^s}),
	\]
while the third summation inside the square brackets can be rewritten as
	\begin{align*}
	\sum_{i+j+p=p^{s+1}} \sigma_i \tau \otimes \sigma_j \tau &= \sum_{\substack{(i+i_0p^s)+(j+j_0p^s)+p=p^{s+1} \\ 0 \leq i,j < p^s,\, 0 \leq i_0,j_0 < p}} \sigma_{i+i_0p^s} \tau \otimes \sigma_{j+j_0p^s} \tau \\
	&= \sum_{\substack{i_0+j_0=p-1 \\ i+j+p=p^s}} \sigma_{i+i_0p^s} \tau \otimes \sigma_{j+j_0p^s} \tau + \sum_{\substack{i_0+j_0=p-2 \\ i+j+p=2p^s,\, i,j \geq 1}} \sigma_{i+i_0p^s} \tau \otimes \sigma_{j+j_0p^s} \tau.
	\end{align*}
Then
	\begin{multline*}
(\rho^* \circ \pi^*)(w_{s+1}) = \rho^* \left( \left[ (-\eta^{-1})^p \cdot \beta(\theta) - \sum_{\substack{i_0+j_0=p-1 \\ i+j=p^s,\, i,j \geq 1}} \frac{(-\eta^{-1})^{p-1}}{(i_0!)(j_0!)} \theta^{i_0} \sigma_i \otimes \theta^{j_0} \sigma_j \right. \right. \\
	\left. \left. - \sum_{\substack{i_0+j_0=p-1 \\ i+j+p=p^s}} \frac{(-\eta^{-1})^{p-1}}{(i_0!)(j_0!)} \theta^{i_0} \sigma_i \tau \otimes \theta^{j_0} \sigma_j \tau - \sum_{\substack{i_0+j_0=p-2 \\ i+j+p=2p^s,\, i,j \geq 1}} \frac{(-\eta^{-1})^{p-2}}{(i_0!)(j_0!)} \theta^{i_0} \sigma_i \tau \otimes \theta^{j_0} \sigma_j \tau \right] \right).
	\end{multline*}
The space $\opH^2(\M_{r;s+1},k)_{\zero}$ has no nonzero elements of internal $\Z$-degrees strictly between $2p^r$ and $2p^{r+(s+1)-1} = 2p^{r+s}$, so when considering the image in $\opH^2(\M_{r;s+1} \otimes_k A,A)$ of the expression above, we can ignore all terms whose internal $\Z$-degrees fall in that range (since the sum of those terms must be cohomologous to zero). In particular, from the three summations in the square brackets we only need to consider contributions of the maximum $\Z$-degree $2p^{r+s}$. Then it follows that
	\begin{flalign*}
	&(\rho^* \circ \pi^*)(w_{s+1}) = \\
	&\left[ (-\eta^{-1})^p \cdot \beta({\textstyle \sum_{i=0}^{r-1}} a_i \theta^{p^i}) + (-\eta^{-1})^p \cdot \beta(-\eta \cdot a_0^{p^{r+s-1}} \sigma_{p^s}) \phantom{\sum_{\substack{i_0+j_0=p-1 \\ i+j=p^s,\, i,j \geq 1}}} \right. \\
	&- \sum_{\substack{i_0+j_0=p-1 \\ i+j+p=p^s}} \frac{(-\eta^{-1})^{p-1}}{(i_0!)(j_0!)} (-\eta a_0^{p^{r+s-1}} \sigma_{p^s})^{i_0} (a_0^{ip^{r-1}} \sigma_i) (\mu \tau) \otimes (-\eta a_0^{p^{r+s-1}} \sigma_{p^s})^{j_0} (a_0^{jp^{r-1}} \sigma_j) (\mu \tau) \\
	&\left. - \sum_{\substack{i_0+j_0=p-2 \\ i+j+p=2p^s,\, i,j \geq 1}} \frac{(-\eta^{-1})^{p-2}}{(i_0!)(j_0!)} (-\eta a_0^{p^{r+s-1}} \sigma_{p^s})^{i_0} (a_0^{ip^{r-1}} \sigma_i) (\mu \tau) \otimes (-\eta a_0^{p^{r+s-1}} \sigma_{p^s})^{j_0} (a_0^{jp^{r-1}} \sigma_j) (\mu \tau) \right] \\
	&= \left[ (-\eta^{-1})^p \cdot \beta({\textstyle \sum_{i=0}^{r-1}} a_i \theta^{p^i}) + a_0^{p^{r+s}} \cdot \beta(\sigma_{p^s}) \phantom{\sum_{\substack{i_0+j_0=p-1 \\ i+j=p^s,\, i,j \geq 1}}} \right. \\
	&\left. - \sum_{\substack{i_0+j_0=p-1 \\ i+j+p=p^s}} a_0^{p^{r+s}} \sigma_{i_0 p^s} \sigma_i \tau \otimes \sigma_{j_0 p^s} \sigma_j \tau - \sum_{\substack{i_0+j_0=p-2 \\ i+j+p=2p^s,\, i,j \geq 1}} a_0^{p^{r+s}} \sigma_{i_0 p^s} \sigma_i \tau \otimes \sigma_{j_0 p^s} \sigma_j \tau \right] \\
	&= (-\eta^{-1})^p \cdot \ol{\beta}([{\textstyle \sum_{i=0}^{r-1}} a_i \theta^{p^i}]) + a_0^{p^{r+s}} \cdot \left[ -\sum_{j=1}^{p^{s+1}-1} \sigma_j \otimes \sigma_{p^{s+1}-j} - \sum_{i+j+p=p^{s+1}} \sigma_i \tau \otimes \sigma_j \tau \right]
	\end{flalign*}
and hence $(\rho^* \circ \pi^*)(w_{s+1}) = (-\eta^{-1})^p \cdot (a_0^p x_1 + \cdots + a_{r-1}^p x_r) + a_0^{p^{r+s}} w_{s+1}$, as desired.
\end{proof}

\begin{corollary} \label{corollary:psirelementary}
Under the identifications and assumptions of Lemma \ref{lemma:kNr}:
	\begin{enumerate}
	\item The homomorphism $\psi_r: H(\Mrone,k) \rightarrow k[\Nr(\Mrone)]$ satisfies
		\[
		\psi_r(y) = \bsmu, \quad \text{and} \quad \psi_r(x_i) = \bsa_{r-i}^{p^i} \quad \text{for $1 \leq i \leq r$.}
		\]
	
	\item Suppose $s \geq 2$. Then the homomorphism $\psi_r: H(\Mrs,k) \rightarrow k[\Nr(\Mrs)]$ satisfies
		\[
		\psi_r(y) = \bsmu, \quad \psi_r(x_i) = \bsa_{r-i}^{p^i} \quad \text{for $1 \leq i \leq r$,} \quad \text{and} \quad \psi_r(w_s) = (\bsb_s)^p.
		\]
	
	\item Suppose $r \geq 2$, let $0 \neq \eta \in k$, and identify $\Hbul(\Mrseta,k)$ as an algebra with $\Hbul(\M_{r-1;s+1},k)$. Then the homomorphism $\psi_r: H(\Mrseta,k) \rightarrow k[\Nr(\Mrseta)]$ satisfies
		\[
		\psi_r(y) = \bsmu, \quad \psi_r(x_i) = \bsa_{r-i-1}^{p^{i+1}} \quad \text{for $1 \leq i \leq r-1$,} \quad \text{and} \quad \psi_r(w_{s+1}) = (-\eta^{-1})^p \cdot \bsa_{r-1}^p.
		\]

	\item The homomorphism $\psi_r: H(\Gar,k) \rightarrow k[\Nr(\Gar)]$ satisfies $\psi_r(x_i) = \bsa_{r-i}^{p^i}$ for $1 \leq i \leq r$.

	\item The homomorphism $\psi_r: H(\Gam,k) \rightarrow k[\Nr(\Gam)]$ satisfies $\psi_r(y) = \bsmu$.
	\end{enumerate}
In general, if $k$ is algebraically closed and if $G$ is a height-$r$ infinitesimal elementary $k$-supergroup scheme, then $\psi_r: H(G,k) \rightarrow k[\NrG]$ is injective modulo nilpotents and its image contains the $p^r$-th powers of all elements in $k[\NrG]$, so the induced map between maximal ideal spectra
	\[
	\Psi : \NrG \rightarrow \abs{G} := \Max(H(G,k)).
	\]
is a homeomorphism of algebraic varieties.
\end{corollary}

\begin{proof}
The five numbered statements follow from the explicit calculations in Lemma \ref{lemma:inducedmapincohomology}, the definition of the map $\psi_r$, and the descriptions in \cite[Proposition 3.2.1]{Drupieski:2017a} and \cite[Lemma 3.2.4]{Drupieski:2017a} of the inflation maps in cohomology arising from the quotient maps $\pi_{r;s}: \Mr \twoheadrightarrow \Mrs$, $q: \Mr \twoheadrightarrow \Gar$, and $q^-: \Mr \twoheadrightarrow \Gam$. The last statement is then an immediate consequence. 
\end{proof}

\subsection{Lyndon--Hochschild--Serre spectral sequences} \label{subsection:LHS}

Let $s \geq 1$. In this section we make some observations concerning various Lyndon--Hochschild--Serre (LHS) spectral sequences that arise from the canonical quotient homomorphisms $\pi_{r;s}: \Mr \twoheadrightarrow \Mrs$, $q: \Mr \twoheadrightarrow \Gar$, and $q^-: \Mr \twoheadrightarrow \Gam$.

\subsubsection{LHS Spectral Sequence for $\pi_{r;s}$} \label{subsubsec:LHSpirs}

As in the proof of \cite[Proposition 3.2.1]{Drupieski:2017a}, let $\Nrs$ be the closed subsupergroup scheme of $\Mr$ defined by $k[\Nrs] = k[\Mr]/\subgrp{\tau,\theta,\sigma_1,\ldots,\sigma_{p^s-1}}$. Then $\Nrs$ is the kernel of the canonical quotient homomorphism $\pi_{r;s} : \Mr \twoheadrightarrow \Mrs$. In particular, $\Nrs$ is normal in $\Mr$. As in \cite[I.6.6]{Jantzen:2003} one gets for each rational $\Mr$-supermodule $V$ an LHS spectral sequence
	\begin{equation} \label{eq:LHSMrs}
	E(V): \quad E_2^{i,j}(V) = \opH^i(\Mrs,\opH^j(\Nrs,V)) \Rightarrow \opH^{i+j}(\Mr,V).
	\end{equation}
Moreover, $E(k)$ is a spectral sequence of algebras, and $E(V)$ is a spectral sequence of right modules over $E(k)$ such that the product on the $E_2$-page, $E_2^{i,j}(V) \otimes E_2^{m,n}(k) \rightarrow E_2^{i+m,j+n}(V)$, identifies up to the sign $(-1)^{j \cdot m}$ with the usual cup product
	\[
	\opH^i(\Mrs,\opH^j(\Nrs,V)) \otimes \opH^m(\Mrs,\opH^n(\Nrs,k)) \rightarrow \opH^{i+m}(\Mrs,\opH^{j+n}(\Nrs,V)).
	\]
Since $\tau$ is an element of the defining ideal of $\Nrs$, it follows that $\Nrs$ is central in $\Mr$. This implies that $\Mrs$ acts trivially on the cohomology ring $\Hbul(\Nrs,k)$, and hence
	\[
	E_2^{i,j}(k) = \opH^i(\Mrs,\opH^j(\Nrs,k)) \cong \opH^i(\Mrs,k) \otimes \opH^j(\Nrs,k).
	\]
As discussed in the proof of \cite[Proposition 3.2.1]{Drupieski:2017a}, $\Hbul(\Nrs,k)$ is an exterior algebra generated by a cohomology class $[\wt{\sigma}_1] \in \opH^1(\Nrs,k)_{\zero}$, and the differential $d_2: E_2^{0,1}(k) \rightarrow E_2^{2,0}(k)$ maps $[\wt{\sigma}_1]$ to $w_s \in \opH^2(\Mrs,k)$, where by convention $w_1 = x_1-y^2$ (cf. \cite[Remark 3.2.2]{Drupieski:2017a}). By the derivation property of the differential, the map $d_2: E_2^{\bullet,1}(k) \rightarrow E_2^{\bullet+2,0}(k)$ then identifies with multiplication by $w_s$, and this is the only nontrivial differential in \eqref{eq:LHSMrs}.

Now suppose $V$ is an $\Mrs$-supermodule, considered as a rational $\Mr$-supermodule via $\pi_{r;s}$. Then $\Nrs$ acts trivially on $V$, and it follows that one has identifications
	\[
	\Hbul(\Mrs,\Hbul(\Nrs,V)) \cong \Hbul(\Mrs,V \otimes \Hbul(\Nrs,k)) \cong \Hbul(\Mrs,V) \otimes \Hbul(\Nrs,k),
	\]
which are compatible with the product structure on the spectral sequence. In particular, $E_2^{i,j}(V) = 0$ for $j \geq 2$, and it follows from the module structure of $E_2(V)$ over $E_2(k)$ that the differential
	\[
	d_2: \opH^\bullet(\Mrs,V) \otimes (k \cdot [\wt{\sigma}_1]) \cong E_2^{\bullet,1}(V) \rightarrow E_2^{\bullet+2,0}(V) = \opH^{\bullet+2}(\Mrs,V)
	\]
identifies with right multiplication on $\Hbul(\Mrs,V)$ by $w_s \in \opH^2(\Mrs,k)$, and this is the only nontrivial differential of the spectral sequence $E(V)$. This description of the differential implies that the kernel of the inflation map $\Hbul(\Mrs,V) \rightarrow \Hbul(\Mr,V)$ (which identifies with the horizontal edge map of \eqref{eq:LHSMrs}) is $\Hbul(\Mrs,V).w_s$, the image of $\Hbul(\Mrs,,V)$ under right multiplication by $w_s$.

\begin{lemma} \label{lemma:Mrfg}
Let $V$ be a finite-dimensional rational $\Mr$-supermodule. Then $\Hbul(\Mr,V)$ is finitely-generated as a (right) $\Hbul(\Mr,k)$-supermodule (or equivalently, as a $H(\Mr,k)$-supermodule).
\end{lemma}

\begin{proof}
First, let $\rho: \Mr \rightarrow GL(V)$ be the homomorphism that affords $V$ the structure of a rational $\Mr$-supermodule. Since $V$ is finite-dimensional, then $GL(V) \cong \GLmn$ for some nonnegative integers $m$ and $n$. In particular, $GL(V)$ is an algebraic $k$-supergroup scheme, so $\rho$ factors for $s \gg 0$ through the quotient map $\pi_{r;s}: \Mr \twoheadrightarrow \Mrs$ by \cite[Remark 3.1.3(4)]{Drupieski:2017a}. Then without loss of generality we may assume that $V$ is an $\Mrs$-supermodule for some fixed integer $s \geq 1$, and we may assume that the $\Mr$-supermodule structure on $V$ is obtained by pulling back along $\pi_{r;s}$.

Now consider the LHS spectral sequence $E_2(V) \Rightarrow \Hbul(\Mr,V)$ as in the paragraph preceding the lemma, and recall from \cite[Theorem 1.1]{Drupieski:2016} that $\Hbul(\Mrs,V)$ is finite as a module under the cup product action of $\Hbul(\Mrs,k)$. Since $\Hbul(\Nrs,k)$ is finite-dimensional, this implies that $E_2(V)$ is finite under the action of $E_2^{\bullet,0}(k) = \Hbul(\Mrs,k)$, and hence implies by \cite[Lemma 1.6]{Friedlander:1997} and \cite[III.2.9 Corollary 1]{Bourbaki:1998} that the inflation map $\Hbul(\Mrs,k) \rightarrow \Hbul(\Mr,k)$ (which identifies with the horizontal edge map of the spectral sequence) makes $\Hbul(\Mr,V)$ into a finitely-generated $\Hbul(\Mrs,k)$-supermodule. In particular, $\Hbul(\Mr,V)$ is finite under the cup product action of $\Hbul(\Mr,k)$. Since $\Hbul(\Mr,k)$ and $H(\Mr,k)$ differ only by nilpotent elements, the last statement is equivalent to saying that $\Hbul(\Mr,V)$ is finite as a $H(\Mr,k)$-supermodule.
\end{proof}

\subsubsection{LHS Spectral Sequence for $q$} \label{subsubsec:LHSq}

The kernel of the quotient homomorphism $q: \Mr \twoheadrightarrow \Gar$ is the closed subsupergroup scheme $Q$ of $\Mr$ defined by $k[Q] = k[\Mr]/\subgrp{\theta}$. Let $\wt{\tau}$ and $\wt{\sigma}_i$ denote the images in $k[Q]$ of the generators $\tau , \sigma_{ip} \in k[\Mr]$, respectively. Then the set $\set{\wt{\sigma}_i, \wt{\sigma}_i \wt{\tau} : i \in \N}$ is a homogeneous basis for $k[Q]$. Now let $T(v) \cong k[v]$ be the tensor (super)algebra over $k$ generated by the element $v$ of odd superdegree. We consider the elements of $k[Q]$ as linear functionals on $T(v)$ such that $\wt{\sigma}_i(v^j) = (-1)^i \cdot \delta_{2i,j}$ and $(\wt{\sigma}_i \wt{\tau})(v^j) = \delta_{2i+1,j}$, where $\delta_{i,j}$ denotes the usual Kronecker delta. In this way, one can then check that the Hochschild complex $C^\bullet(Q,k) = k[Q]^{\otimes \bullet}$ for $Q$ identifies with the cobar complex for $T(v)$ as defined in \cite{Priddy:1970}. Then by \cite[Theorem 2.5]{Priddy:1970}, the cohomology ring $\Hbul(Q,k)$ is the exterior algebra over $k$ generated by the cohomology class $[\wt{\tau}] \in \opH^1(Q,k)_{\one}$, i.e., $\Hbul(Q,k) \cong \Lambda([\wt{\tau}])$. Since the generator $y \in \Hbul(\Mr,k)$ is by definition the cohomology class of the cochain $\tau \in k[\Mr]$, it immediately follows that the restriction map in cohomology $\Hbul(\Mr,k) \rightarrow \Hbul(Q,k)$ is surjective and sends $y$ to $[\wt{\tau}]$.

Now consider the LHS spectral sequence for the group extension $Q \hookrightarrow \Mr \stackrel{q}{\twoheadrightarrow} \Gar$:
	\begin{equation} \label{eq:LHSGar}
	E_2^{i,j}(k) = \opH^i(\Gar,\opH^j(Q,k)) \Rightarrow \opH^{i+j}(\Mr,k).
	\end{equation}
Since $\opH^0(Q,k)$ and $\opH^1(Q,k)$ are each one-dimensional, and since $\Gar$ is unipotent, it follows that $\Gar$ acts trivially on $\Hbul(Q,k)$. Then $E_2^{i,j}(k) \cong \opH^i(\Gar,k) \otimes \opH^j(Q,k)$, and the only nontrivial differential in \eqref{eq:LHSGar} is $d_2: E_2^{\bullet,1} \rightarrow E_2^{\bullet+2,0}$. The horizontal edge map of \eqref{eq:LHSGar} identifies with the inflation map $\Hbul(\Gar,k) \rightarrow \Hbul(\Mr,k)$, and so is injective by \cite[Lemma 2.3.4]{Drupieski:2017a}. Then $d_2 \equiv 0$.

Finally, let $V$ be a finite-dimensional $\Gar$-supermodule, considered as a rational $\Mr$-super\-module via the quotient $q: \Mr \twoheadrightarrow \Gar$. Then as in Section \ref{subsubsec:LHSpirs} above, the $E_2$-page of the LHS spectral sequence $E_2^{i,j}(V) = \opH^i(\Gar,\opH^j(Q,V)) \Rightarrow \opH^{i+j}(\Mr,V)$ admits a factorization
	\[
	E_2^{i,j}(V) \cong \opH^i(\Gar,V) \otimes \opH^j(Q,k).
	\]
Using this, the right module structure of $E_2(V)$ over $E_2(k)$, the derivation property of the differential, and the fact that the differential on $E_2(k)$ is trivial, it follows that the differential on $E_2(V)$ is trivial as well. Consequently, the inflation map $\Hbul(\Gar,V) \rightarrow \Hbul(\Mr,V)$ is an injection.

\subsubsection{LHS Spectral Sequence for $q^-$} \label{subsubsec:LHSq-}

The kernel of the quotient homomorphism $q^-: \Mr \twoheadrightarrow \Gam$ is the closed subsupergroup scheme $Q^-$ of $\Mr$ defined by $k[Q^-] = k[\Mr]/\subgrp{\tau}$. In particular, $Q^-$ is abelian because $k[Q^-]$ is cocommutative in the sense of Hopf superalgebras. Let $\wt{\theta}$ and $\wt{\sigma}_i$ denote the images in $k[Q^-]$ of the generators $\theta , \sigma_i \in k[\Mr]$, respectively. The supergroup $\Nrone$ defined in Section \ref{subsubsec:LHSpirs} is a (normal) subsupergroup of $Q^-$ because $k[Q^-]/\subgrp{\wt{\theta}} \cong k[\Mr]/\subgrp{\theta,\tau} = k[\Nrone]$, and $Q^-/\Nrone \cong \Gar$. Consider the LHS spectral sequence for the extension $\Nrone \hookrightarrow Q^- \twoheadrightarrow \Gar$:
	\begin{equation} \label{eq:LHSQ-}
	E_2^{i,j} = \opH^i(\Gar,\opH^j(\Nrone,k)) \Rightarrow \opH^{i+j}(Q^-,k).
	\end{equation}
Since $Q^-$ is abelian, the action of $Q^-/\Nrone \cong \Gar$ on $\Hbul(\Nrone,k)$ (induced by conjugation) is trivial, and hence $E_2^{i,j} \cong \opH^i(\Gar,k) \otimes \opH^j(\Nrone,k)$. From the discussion in Section \ref{subsubsec:LHSpirs} above, the cohomology ring $\Hbul(\Nrone,k)$ is an exterior algebra generated by the cohomology class $[\wh{\sigma}] \in \opH^1(\Nrone,k)$, where $\wh{\sigma}$ denotes the image of $\wt{\sigma}_p \in k[Q^-]$ under the quotient map $k[Q^-] \twoheadrightarrow k[Q^-]/\subgrp{\wh{\theta}} = k[\Nrone]$. The differential $d_2: E_2^{0,1} \rightarrow E_2^{2,0}$ is induced by the differential $\partial$ on the Hochschild complex for $Q^-$. Then in terms of the map $\beta$ of \eqref{eq:beta},
	\[ \textstyle
	d_2([\wh{\sigma}]) = [\partial(\wt{\sigma}_p)] = [-\sum_{i=1}^{p-1} \wt{\sigma}_i \otimes \wt{\sigma}_{p-i}] = [\beta(\wt{\theta}^{p^{r-1}})] = x_r \in \opH^2(\Gar,k).
	\]
Since $d_2: E_2^{\bullet,1} \rightarrow E_2^{\bullet+2,0}$ is the only nontrivial differential in \eqref{eq:LHSQ-}, and since $x_r$ is not a zero divisor in $\Hbul(\Gar,k)$, this implies that the inflation map $\Hbul(\Gar,k) \rightarrow \Hbul(Q^-,k)$, which identifies with the horizontal edge map of \eqref{eq:LHSQ-}, is surjective with kernel generated by $x_r$, i.e.,
	\[
	\Hbul(Q^-,k) \cong k[x_1,x_2,\ldots,x_{r-1}] \gotimes \Lambda(\lambda,\ldots,\lambda_r).
	\]
Moreover, replacing $\theta$ by $\wt{\theta}$, cochain representatives for the generators of $\Hbul(Q^-,k)$ are given by the same formulas as in \cite[Proposition 3.2.1]{Drupieski:2017a}.

Now consider the LHS spectral sequence (with trivial coefficients) for the supergroup extension $Q^- \hookrightarrow \Mr \twoheadrightarrow \Gam$. Since $\tau$ is an element of the defining ideal of $Q^-$, it follows that $Q^-$ is central in $\Mr$, and hence the spectral sequence takes the form
	\begin{equation} \label{eq:LHSGa-}
	E_2^{i,j} \cong \opH^i(\Gam,k) \otimes \opH^j(Q^-,k) \Rightarrow \opH^{i+j}(\Mr,k).
	\end{equation}
This spectral sequence is not confined simply to two rows, but it follows from inspecting the explicit formulas for the cochain representatives of the generators that the restriction map $\Hbul(\Mr,k) \rightarrow \Hbul(Q^-,k)$, which identifies with the vertical edge map of \eqref{eq:LHSGa-}, is surjective with kernel generated by $y \in \Hbul(\Mr,k)$. Then the column $E_2^{0,\bullet}$ in \eqref{eq:LHSGa-} consists of permanent cycles, which implies by the derivation property of the differential that all differentials in \eqref{eq:LHSGa-} are zero.

Finally, let $V$ be a finite-dimensional $\Gam$-supermodule, considered as a rational $\Mr$-supermodule via the quotient $q^-: \Mr \rightarrow \Gam$. Then by reasoning precisely analogous to that in Section \ref{subsubsec:LHSq} above, it follows that the LHS spectral sequence
	\[
	E_2^{i,j}(V) = \opH^i(\Gam,\opH^j(Q^-,V)) \Rightarrow \opH^{i+j}(\Mr,V)
	\]
admits the factorization $E_2^{i,j}(V) \cong \opH^i(\Gam,V) \otimes \opH^j(Q^-,k)$ on the $E_2$-page, the differentials are all trivial, and the inflation map $\Hbul(\Gam,V) \rightarrow \Hbul(\Mr,V)$ is an injection.

\subsection{Support sets for modules} \label{subsection:supportset}

\begin{definition}
Let $G$ be an affine $k$-supergroup scheme. Given a rational $G$-super\-module $V$, define the injective dimension of $V$ (in the category of rational $G$-super\-modules) by
	\[
	\id_G(V) = \inf \set{ n \geq 0 : \Ext_G^i(-,V) \equiv 0 \text{ for all } i > n}.
	\]
\end{definition}

\begin{lemma} \label{lemma:idequivalent}
Let $G$ be an affine $k$-supergroup scheme, and let $V$ be a rational $G$-supermodule. Then the following are equivalent:
	\begin{enumerate}
	\item $\id_G(V) \leq n < \infty$
	\item \label{item:extvanishing} $\Ext_G^i(L,V) = 0$ for all $i > n$ and all irreducible rational $G$-supermodules $L$.
	\item There exists an injective resolution of $V$,
		\[
		0 \rightarrow V \rightarrow Q_0 \rightarrow Q_1 \rightarrow Q_2 \rightarrow \cdots,
		\]
	such that $Q_i = 0$ for all $i > n$.
	\end{enumerate}
\end{lemma}

\begin{proof}
The first condition evidently implies the second, and the third condition evidently implies the first, so suppose that $\Ext_G^i(L,V) = 0$ for all $i > n$ and all irreducible rational $G$-supermodules $L$. Let $Q_\bullet$ be a minimal injective resolution of $V$. The minimality of $Q_\bullet$ implies for each irreducible $G$-supermodule $L$ that $\Ext_G^i(L,V) = \Hom_G(L,Q_i)$, and hence $\Hom_G(L,Q_i) = 0$ for all $i > n$ and all irreducibles $L$. Then $\soc_G(Q_i) = 0$ for $i > n$, which implies that $Q_i = 0$ for $i > n$.
\end{proof}

Lemma \ref{lemma:Mrquotient} and the local finiteness of rational representations imply that $\Mr$ is unipotent, and hence that up to isomorphism and parity flip, the only irreducible $\Mr$-supermodule is the trivial module $k$. Thus for $G = \Mr$, it suffices to consider just $L = k$ in Lemma \ref{lemma:idequivalent}\eqref{item:extvanishing}.

\begin{lemma} \label{lemma:idEndV}
Let $G$ be an affine $k$-supergroup scheme, and let $V$ be a finite-dimensional rational $G$-supermodule. Then
	\[
	\id_G(V) = \id_G(V \otimes V^\#) = \id_G(\Hom_k(V,V)).
	\]
\end{lemma}

\begin{proof}
First observe that since $V$ is finite-dimensional, the natural map $\alpha: V \otimes V^\# \rightarrow \Hom_k(V,V)$, defined for $v,v' \in V$ and $\phi \in V^\#$ by $\alpha(v \otimes \phi)(v')= \phi(v') \cdot v$, is a $G$-supermodule isomorphism. Then it suffices to show that  $\id_G(V \otimes V^\#) \leq \id_G(V)$ and $\id_G(V) \leq \id_G(V \otimes V^\#)$. To prove the first inequality, suppose $\id_G(V) = n < \infty$. Then by Lemma \ref{lemma:idequivalent}, there exists an injective resolution $Q_\bullet$ of $V$ with the property that $Q_i = 0$ for $i > n$. Then $V^\# \otimes Q_\bullet$ is an injective resolution of $V^\# \otimes V$ (cf.\ \cite[I.3.10]{Jantzen:2003}). The supertwist map $T: V \otimes V^\# \rightarrow V^\# \otimes V$, defined by $T(v \otimes \phi) = (-1)^{\ol{v} \cdot \ol{\phi}} \phi \otimes v$, is a $G$-supermodule isomorphism, so we deduce that there exists an injective resolution $Q_\bullet'$ of $V \otimes V^\# \cong V^\# \otimes V$ such that $Q_i' = 0$ for $i > n$. Then $\id_G(V \otimes V^\#) \leq n = \id_G(V)$.

For the other inequality, suppose $\id_G(V \otimes V^\#) = n < \infty$. Then by the same type of reasoning as in the previous paragraph, $\id_G((V \otimes V^\#) \otimes V) \leq n < \infty$, and hence $\Ext_G^i(L,V \otimes V^\# \otimes V) = 0$ for all $i > n$ and all irreducible $G$-supermodules $L$. Define $\tau: k \rightarrow \Hom_k(V,V)$ by $\tau(\lambda) = \lambda \cdot 1_V$, and let $c: V^\# \otimes V \rightarrow k$ be the contraction map defined by $c(\phi \otimes v) = \phi(v)$. It is straightforward to check that $\tau$ and $c$ are both $G$-supermodule homomorphisms. If $\set{v_1,\ldots,v_m}$ is a fixed homogeneous basis for $V$, and if $\set{v_1^*,\ldots,v_m^*}$ is the corresponding dual basis such that $v_i^*(v_j) = \delta_{i,j}$, then under the identification $\Hom_k(V,V) \cong V \otimes V^\#$ the map $\tau$ satisfies $\tau(\lambda) = \lambda \cdot (\sum_{i=1}^m v_i \otimes v_i^*)$. Now consider the composite $G$-supermodule homomorphism
	\[
	V \cong k \otimes V \stackrel{\tau \otimes 1}{\longrightarrow} V \otimes V^\# \otimes V \stackrel{1 \otimes c}{\longrightarrow} V \otimes k \cong V.
	\]
It is straightforward to check (e.g., by first considering the effect on basis vectors and then extending linearly) that this composite is the identity. Consequently, $V$ is a $G$-supermodule direct summand of $V \otimes V^\# \otimes V$. Then the equality $\Ext_G^i(L,V \otimes V^\# \otimes V) = 0$ for $i > n$ implies that $\Ext_G^i(L,V) = 0$ for $i > n$, and hence implies that $\id_G(V) \leq n = \id_G(V \otimes V^\#)$.
\end{proof}

Given a supergroup scheme homomorphism $\phi: G \rightarrow H$ and a rational $H$-supermodule $V$, let $\phi^* V$ denote the rational $G$-supermodule obtained from $V$ by pulling back along $\phi$.

\begin{definition} \label{definition:supportset}
Let $G$ be an algebraic $k$-supergroup scheme. Given a rational $G$-supermodule $M$, define the support set $\NoneG_M \subset \NoneG$ by
	\[
	\NoneG_M = \set{ \phi \in \NoneG = \bfHom(\Mone,G)(k) : \id_{\Mone}(\phi^* M) = \infty}.
	\]
\end{definition}

Applying Lemma \ref{lemma:idequivalent} and the observation immediately following it, Lemma \ref{lemma:idEndV}, and the adjoint associativity isomorphism, it follows that if $M$ is finite-dimensional, then
	\[
	\NoneG_M = \set{ \phi \in \NoneG : \Ext_{\Mone}^i(\phi^*M,\phi^*M) \neq 0 \text{ for infinitely many $i \geq 1$}}.
	\]
In Section \ref{subsection:supportheightone} we will show that if $G$ is a height-one infinitesimal elementary supergroup scheme and if $M$ is a finite-dimensional $G$-supermodule, then $\NoneG_M$ is closed in $\NoneG$ and identifies with the cohomological support variety $\abs{G}_M$. More generally, in Section \ref{subsection:BIKPapplications} we will discuss how we expect that results recently announced by Benson, Iyengar, Krause, and Pevtsova (BIKP) can be applied to obtain more general results concerning the cohomological spectrum $\abs{G}$ and the support varieties $\abs{G}_M$ when $G$ is a more general infinitesimal unipotent supergroup scheme.

\begin{question} \label{question:supportsetclosed}
Suppose $k$ is algebraically closed. Let $G$ be an arbitrary height-$1$ infinitesimal $k$-supergroup scheme, and let $M$ be a finite-dimensional $G$-supermodule. Is $\NoneG_M$ Zariski closed in $\NoneG$?
\end{question}

As mentioned in the introduction to the paper, our definition of the support set $\NoneG_M$ is inspired by similar definitions appearing in the literature in the context of commutative local rings \cite{Avramov:1989,Avramov:2000,Jorgensen:2002}, and which were brought to our attention by way of a talk by Srikanth Iyengar. At present we do not understand how the existing support theory can be related to the rational cohomology of $\Mone$, but some obvious intermediate actors are the group algebra
	\[
	k\Mone = k[[u,v]]/\subgrp{u^p+v^2}
	\]
and its polynomial subalgebra
	\[
	\Pone = k[u,v]/\subgrp{u^p+v^2}.
	\]
In the rest of this section we will consider the more general question of the relationship between rational cohomology for $\Mr$ and cohomology for its polynomial subalgebra $\Pr$.

Let ${}_{\Mr}\fsmod$ denote the category of rational left $\Mr$-supermodules, and let ${}_{\Pr}\fsmod$ denote the category of left $\Pr$-supermodules. It follows from Proposition \ref{proposition:HomMrGLmn} that ${}_{\Mr}\fsmod$ identifies with the full subcategory of ${}_{\Pr}\fsmod$ consisting of the objects on which the generator $u_{r-1} \in \Pr$ acts via a locally nilpotent endomorphism. The exact inclusion of categories ${}_{\Mr}\fsmod \hookrightarrow {}_{\Pr}\fsmod$ then induces for each pair of rational $\Mr$-supermodules $M$ and $N$ a linear map on cohomology groups
	\begin{equation} \label{eq:ExtMrmap}
	\Ext_{\Mr}^\bullet(M,N) \rightarrow \Ext_{\Pr}^\bullet(M,N).
	\end{equation}
We will show that, at least when $M$ and $N$ are finite-dimensional, \eqref{eq:ExtMrmap} is an isomorphism.

\begin{proposition} \label{proposition:rationaltoarbitrary}
Let $M$ and $N$ be finite-dimensional rational $\Mr$-supermodules. Then the map in cohomology \eqref{eq:ExtMrmap} induced by the inclusion ${}_{\Mr}\fsmod \hookrightarrow {}_{\Pr}\fsmod$ is an isomorphism.
\end{proposition}

\begin{proof}
Set $V = N \otimes M^\#$. Then by the usual adjoint associativity isomorphisms (cf.\ \cite[I.4.4]{Jantzen:2003} and \cite[Lemma 2.3.4]{Drupieski:2016a}), it suffices to show that the map $\Hbul(\Mr,V) \rightarrow \Hbul(\Pr,V)$ is an isomorphism. Let $\Delta_V: V \rightarrow V \otimes k[\Mr]$ be the comodule map that defines the rational $\Mr$-supermodule structure on $V$. Since $V$ is finite-dimensional, there exists an integer $s \geq 1$ such that the image of $\Delta_V$ is contained in $V \otimes k[\Mrs]$, and hence such that the action of $\Mr$ on $V$ factors through the quotient $\Mr \twoheadrightarrow \Mr/\Nrs = \Mrs$. Set $A = \Pr$, and let $B = k[x]$ be the Hopf subalgebra of $A$ generated by $x:= u_{r-1}^{p^s}$. Then $B$ is a central, primitively-generated Hopf subalgebra of $A$, and the Hopf quotient $A \sslash B = A/\subgrp{x}$ is just $\Pr/\subgrp{u_{r-1}^{p^s}} = k\Mrs$. Since $\Nrs$ acts trivially on $V$, it follows that $B$ does as well, and hence that the action of $\Pr$ on $V$ factors through the quotient $A \sslash B$. Now our strategy for showing that the map $\Hbul(\Mr,V) \rightarrow \Hbul(\Pr,V)$ is an isomorphism will be to show that the inclusion of categories ${}_{\Mr}\fsmod \hookrightarrow {}_{\Pr}\fsmod$ induces a morphism of spectral sequences between (i) the LHS spectral sequence of rational cohomology groups for the group extension
	\begin{equation} \label{eq:groupextension}
	1 \rightarrow \Nrs \rightarrow \Mr \rightarrow \Mrs \rightarrow 1,
	\end{equation}
and (ii) the LHS spectral sequence for the extension of Hopf superalgebras
	\begin{equation} \label{eq:Hopfextension}
	0 \rightarrow B \rightarrow A \rightarrow A\sslash B \rightarrow 0.
	\end{equation}
We'll show that this morphism is an isomorphism on the $E_2$-page, and hence induces a morphism between the abutments $\Hbul(\Mr,V)$ and $\Hbul(\Pr,V)$.

Let $V \rightarrow I^\bullet$ be a resolution of $V$ by injective rational $\Mr$-supermodules, and let $V \rightarrow Q^\bullet$ be a resolution of $V$ by injective $\Pr$-supermodules. Then $V \rightarrow I^\bullet$ remains exact as a complex of $\Pr$-super\-modules, so there exists a $\Pr$-supermodule chain homomorphism $I^\bullet \rightarrow Q^\bullet$ that lifts the identity on $V$. For $W$ a rational $\Mr$-supermodule, considered also as a $\Pr$-supermodule, one has $W^B = W^{\Nrs}$. Then the map $I^\bullet \rightarrow Q^\bullet$ restricts to a chain map $(I^\bullet)^{\Nrs} \rightarrow (Q^\bullet)^B$. Since $\Mr/\Nrs = \Mrs$ and $A \sslash B = k\Mrs$, and since $\Mrs$ is a finite $k$-supergroup scheme, we can consider $(I^\bullet)^{\Nrs}$ and $(Q^\bullet)^B$ both either as complexes of rational $\Mrs$-supermodules or as complexes of $k\Mrs$-supermodules. So let $(I^\bullet)^{\Nrs} \rightarrow J^{\bullet,\bullet}$ and $(Q^\bullet)^B \rightarrow K^{\bullet,\bullet}$ be Cartan-Eilenberg resolutions of $(I^\bullet)^{\Nrs}$ and $(Q^\bullet)^B$, respectively, in the category of $k\Mrs$-supermodules. Then the chain map $(I^\bullet)^{\Nrs} \rightarrow (Q^\bullet)^B$ lifts to a double complex chain map $J^{\bullet,\bullet} \rightarrow K^{\bullet,\bullet}$ of $k\Mrs$-supermodules. Now the LHS spectral sequence of rational cohomology groups for the group extension \eqref{eq:groupextension},
	\[
	E_2^{i,j}(\Mr) = \opH^i(\Mrs,\opH^j(\Nrs,V)) \Rightarrow \opH^{i+j}(\Mr,V),
	\]
arises from the double complex $(J^{\bullet,\bullet})^{\Mrs}$ by first computing cohomology along columns and then along rows, while the LHS spectral sequence for the extension of Hopf superalgebras \eqref{eq:Hopfextension},
	\[
	E_2^{i,j}(\Pr) = \opH^i(A\sslash B,\opH^j(B,V)) \Rightarrow \opH^{i+j}(A,V),
	\]
arises from the double complex $(K^{\bullet,\bullet})^B$ in the same manner. The chain map $J^{\bullet,\bullet} \rightarrow K^{\bullet,\bullet}$ restricts to a double complex chain map $(J^{\bullet,\bullet})^{\Mrs} \rightarrow (K^{\bullet,\bullet})^{k\Mrs}$, and hence we get a morphism of spectral sequences $E(\Mr) \rightarrow E(\Pr)$. On the $E_2$-page, the map
	\[
	\opH^i(\Mrs,\opH^j(\Nrs,V)) \rightarrow \opH^i(k\Mrs,\opH^j(B,V))
	\]
is induced by the equivalence of categories ${}_{\Mrs}\fsmod \cong {}_{k\Mrs}\fsmod$ and by the inclusion of categories ${}_{\Nrs}\fsmod \hookrightarrow {}_{B}\fsmod$, the latter of which gives rise to a homomorphism of $\Mrs$-supermodules (equivalently, of $k\Mrs$-supermodules) $\opH^j(\Nrs,V) \rightarrow \opH^j(B,V)$. Thus to show that the morphism of spectral sequences $E(\Mr) \rightarrow E(\Pr)$ is an isomorphism from the $E_2$-page onwards, and hence show that the abutments $\Hbul(\Mr,V)$ and $\Hbul(\Pr,V)$ are isomorphic, it suffices to show that the inclusion of categories ${}_{\Nrs}\fsmod \hookrightarrow {}_{B}\fsmod$ induces an isomorphism $\Hbul(\Nrs,V) \cong \Hbul(B,V)$. Since $\Nrs$ and $B$ each act trivially on $V$, we have $\Hbul(\Nrs,V) \cong \Hbul(\Nrs,k) \otimes V$ and $\Hbul(B,V) \cong \Hbul(B,k) \otimes V$, so it suffices even to assume that $V = k$. In this case we know that $\Hbul(\Nrs,k)$ and $\Hbul(B,k) = \Hbul(k[x],k)$ are both exterior algebras generated in cohomological degree one, and one can check that the generator $[\wt{\sigma}_1] \in \opH^1(\Nrs,k)$ gets mapped to the generator in $\opH^1(k[x],k)$. Thus $\Hbul(\Nrs,k) \cong \Hbul(B,k)$.
\end{proof}

\subsection{Support varieties in height one} \label{subsection:supportheightone}

In this section we describe the support variety $\abs{G}_M$ of a finite-dimensional rational $G$-module $M$ in the case when $G$ is a height-one infinitesimal elementary supergroup scheme, i.e., when $G$ is one of $\Gaone$, $\Gam$, or $\Mones$ for some $s \geq 1$. \emph{Throughout this section assume that $k$ is an algebraically closed field.}

Let $G$ be a finite supergroup scheme. Recall that the cohomological spectrum $\abs{G}$ of $G$ is the maximal ideal spectrum of the cohomology ring $\Hbul(G,k)$, or equivalently (cf.\ \cite[Corollary 2.2.5]{Drupieski:2016a}) the maximal ideal spectrum of the subring $H(G,k)$ defined in \eqref{eq:HGk}:
	\[
	\abs{G} := \Max\left( \Hbul(G,k) \right) = \Max\left( H(G,k) \right).
	\]
Now let $M$ be a finite-dimensional rational $G$-supermodule, and set $\Lambda = \Hom_k(M,M)$. Then $\Lambda$ is a finite-dimensional unital rational $G$-algebra, and the unit map $1_\Lambda: k \rightarrow \Lambda$ is a $G$-supermodule homomorphism. Let $\rho_\Lambda: \Hbul(G,k) \rightarrow \Hbul(G,\Lambda)$ be the map in cohomology induced by $1_\Lambda$, and set $I_M = \ker(\rho_\Lambda)$. By abuse of notation, we also denote the intersection $I_M \cap H(G,k)$ by $I_M$. Then the cohomological support variety $\abs{G}_M$ is the Zariski closed subvariety of $\abs{G}$ defined by $I_M$:
	\[
	\abs{G}_M := \Max\left( \Hbul(G,k)/I_M \right) = \Max\left( H(G,k)/I_M \right) = \Max\left( H(G,k)/ \sqrt{I_M} \right).
	\]
Using the adjoint associativity isomorphism
	\[
	\Hbul(G,\Lambda) = \Ext_G^\bullet(k,\Hom_k(M,M)) \cong \Ext_G^\bullet(k,M \otimes M^\#) \cong \Ext_G^\bullet(M,M),
	\]
one can check that this definition for $\abs{G}_M$ agrees with \cite[Definition 2.3.8]{Drupieski:2016a}. Letting $1_M \in \Hom_G(M,M) \cong \opH^0(G,\Lambda)$ denote the identity map on $M$, one can check that $\rho_\Lambda(z) = 1_M \cdot z$, the image of $1_M$ under the right cup product action of $z$.

Now let $G$ be a height-one infinitesimal elementary $k$-supergroup scheme, and let $M$ be a finite-dimensional rational $G$-supermodule. By Corollary \ref{corollary:psirelementary}, the map $\psi_1: H(G,k) \rightarrow k[\NoneG]$ induces a homeomorphism of varieties $\Psi: \NoneG \rightarrow \abs{G}$. The main result of this section is to relate the cohomological support variety $\abs{G}_M$ to the support set $\NoneG_M$ defined in Definition \ref{definition:supportset}.

\begin{theorem} \label{theorem:supportheightone}
Suppose $k$ is algebraically closed. Let $G$ be a height-one infinitesimal elementary $k$-supergroup scheme, and let $M$ be a finite-dimensional rational $G$-supermodule. Then
	\[
	\Psi^{-1}(\abs{G}_M) = \NoneG_M = \set{ \phi \in \NoneG : \id_{\Mone}(\phi^* M) = \infty}
	\]
In particular, $\NoneG_M$ is closed in $\NoneG$.
\end{theorem}

We break the proof up into cases, depending on the isomorphism type of $G$.

\subsubsection{Proof of Theorem \ref{theorem:supportheightone} for $G = \Mones$ and $s \geq 2$} \label{subsubsec:proofMones}

In this case
	\[
	H(\Mones,k) \cong k[x_1,y,w_s]/\subgrp{x_1-y^2} \cong k[y,w_s].
	\]
The points of the cohomological spectrum $\abs{\Mones}$ may be identified with $k$-algebra homomorphisms $H(\Mones,k) \rightarrow k$. Given scalars $\mu,a,b \in k$ such that $\mu^2 = a$, let $\wh{\phi}_{(\mu,a,b)} : H(\Mones,k) \rightarrow k$ denote the homomorphism such that $\wh{\phi}_{(\mu,a,b)}(y) = \mu$, $\wh{\phi}_{(\mu,a,b)}(x_1) = a$, and $\wh{\phi}_{(\mu,a,b)}(w_s) = b$. Similarly, given scalars $\mu,a,b \in k$ such that $\mu^2 = a^p$, let $\phi_{(\mu,a,b)}: \Mone \rightarrow \Mones$ denote the homomorphism labeled by $(\mu,a,b)$ as in Proposition \ref{prop:identifyhom}\eqref{item:HomMrs}. Then Corollary \ref{corollary:psirelementary} implies that $\Psi(\phi_{(\mu,a,b)}) = \wh{\phi}_{(\mu,a^p,b^p)}$.

Observe that $\Aut(\Mones) = \bfAut(\Mones)(k)$ acts on $\None(\Mones)$ via composition, $\nu_*: \phi \mapsto \nu \circ \phi$, and acts on $\abs{\Mones}$ by the $\nu_*(\wh{\phi}) = \wh{\phi} \circ \nu^*: H(\Mones,k) \rightarrow H(\Mones,k) \rightarrow k$. Moreover, it follows from the last assertion of Proposition \ref{prop:psirgradedmap} that $\Psi$ is $\Aut(\Mones)$-equivariant. Given scalars $\lambda,c,d \in k$ such that $\lambda^2 = c^p \neq 0$, let $\nu = \nu_{(\lambda,c,d)}: \Mones \rightarrow \Mones$ be the corresponding automorphism as in Lemma \ref{lemma:automorphisms}. Then an elementary computation (noting that comorphisms compose in the reverse order from homomorphisms) and Lemma \ref{lemma:inducedmapincohomology}\eqref{item:rho*Mrs} imply that
	\[
	\nu_*(\phi_{(\mu,a,b)}) = \phi_{(\mu\lambda,ac,bc^{p^{s-1}}+ad)} \quad \text{and} \quad \nu_*(\wh{\phi}_{(\mu,a,b)}) = \wh{\phi}_{(\mu\lambda,ac^p,bc^{p^s}+ad^p)}.
	\]
Furthermore, the action of the automorphism $\nu$ on $\abs{G}$ maps $\abs{G}_{\nu^*M}$ isomorphically onto $\abs{G}_{M}$. Now, there are precisely two orbits of $\Aut(\Mones)$ in $\None(\Mones)$: the orbit of the canonical quotient map $\pi_{1;s} = \phi_{(1,1,0)}: \Mone \rightarrow \Mones$ and the orbit of $\phi_{(0,0,1)}: \Mone \rightarrow \Mones$. Then to finish the proof of the theorem in the case $G = \Mones$ with $s \geq 2$, it suffices to show that
	\begin{enumerate}
	\item $\wh{\phi}_{(1,1,0)} \in \abs{\Mones}_M$ if and only if $\phi_{(1,1,0)} \in \None(\Mones)_M$, and
	\item $\wh{\phi}_{(0,0,1)} \in \abs{\Mones}_M$ if and only if $\phi_{(0,0,1)} \in \None(\Mones)_M$.
	\end{enumerate}

First we prove statement (1). Set $\Lambda = \Hom_k(M,M)$, and set $\pi = \pi_{1;s} = \phi_{(1,1,0)}$. Then there exists a commutative diagram
\[
\xymatrix@C+.5em{
H(\Mones,k) \ar@{->}[r]^{\rho_{\Lambda}} \ar@{->}[d]^{\pi^*} & \Hbul(\Mones,\Lambda) \ar@{->}[d]^{\pi^*} \\
H(\Mone,k) \ar@{->}[r]^{\rho_{\pi^* \Lambda}} & \Hbul(\Mone,\pi^* \Lambda)
}
\]
in which the horizontal arrows are induced by the unit map $k \rightarrow \Lambda$, and the vertical arrows are the inflation maps induced by $\pi$. First suppose $\pi \in \None(\Mones)_M$, so that $\opH^i(\Mone,\pi^*\Lambda) \neq 0$ for infinitely many $i \geq 0$. Since $\Hbul(\Mone,\pi^* \Lambda)$ is finitely generated over $H(\Mone,k) \cong k[y]$ by Lemma \ref{lemma:Mrfg}, this implies that $y^n \notin \ker(\rho_{\pi^*\Lambda})$ for all $n \geq 1$. Then the commutativity of the diagram implies that $\ker(\rho_\Lambda)$ is contained in the kernel of the inflation map $\pi^*: H(\Mones,k) \rightarrow H(\Mone,k)$, i.e., $I_M \subseteq \subgrp{w_s}$, which in turn implies that $\wh{\pi} = \wh{\phi}_{(1,1,0)} \in \abs{G}_M$.

For the converse of (1), suppose $\pi \notin \None(\Mones)_M$. We want to show that $\wh{\pi} \notin \abs{G}_M$, or equivalently that $\Hbul(\Mones,\Lambda)_{\fm} = 0$. Here $\Hbul(\Mones,\Lambda)_{\fm}$ denotes the localization of $\Hbul(\Mones,\Lambda)$ at the maximal homogeneous ideal $\fm = \subgrp{w_s}$ that defines the line in $\abs{\Mones}$ through $\wh{\pi}$. Since $\pi \notin \None(\Mones)_M$, there exists an integer $N \geq 1$ such that $\opH^i(\Mone,\pi^*\Lambda) = 0$ for $i \geq N$. Then $(\pi^* \circ \rho_\Lambda)(y^N) = 0$. By the discussion in Section \ref{subsubsec:LHSpirs}, this implies that $\rho_\Lambda(y^N) = z.w_s$ for some $z \in \opH^{N-2}(\Mones,\Lambda)$. In other words, $1_M . y^N = z.w_s$ in $\Hbul(\Mones,\Lambda)$. Now consider the homogeneous localizations of $\Hbul(\Mones,k)$ and $\Hbul(\Mones,\Lambda)$ at the maximal homogeneous ideal $\fm = \subgrp{w_s}$. Since $y$ is a unit in $\Hbul(\Mones,k)_{\fm}$, the identity $1_M \cdot y^N = z.w_s$ implies that $\Hbul(\Mones,\Lambda)_{\fm} = (\Hbul(\Mones,\Lambda)_{\fm}).\fm$. Then the graded analogue of Nakayama's Lemma \cite[Ex.\ 1.5.24]{Bruns:1993} implies that $\Hbul(\Mones,\Lambda)_{\fm} = 0$, and hence that $\wh{\pi} \notin \abs{G}_M$.

Now we prove (2). Set $\phi = \phi_{(0,0,1)}$, and write $\phi = \rho \circ \pi: \Mone \twoheadrightarrow \Mones \rightarrow \Mones$, where $\rho = \rho_{(0,0,1)}$ as in Remark \ref{rem:Hom}\eqref{item:iota}. First suppose $\phi \in \None(\Mones)_M$. Then $\opH^i(\Mone,\phi^* \Lambda) = \opH^i(\Mone,\pi^*(\rho^*\Lambda))$ is nonzero for infinitely many $i \geq 0$, which means that $\pi \in \None(\Mones)_{\rho^*M}$, and hence $\wh{\pi} \in \abs{\Mones}_{\rho^*M}$ by (1). The homomorphism $\rho: \Mones \rightarrow \Mones$ induces a morphism of varieties $\rho_*: \abs{\Mones}_{\rho^*M} \rightarrow \abs{\Mones}_M$, and one can check using Lemma \ref{lemma:inducedmapincohomology}\eqref{item:rho*Mrs} that $\rho_*(\wh{\pi}) = \wh{\phi}$. Now for the converse of (2), suppose that $\phi \notin \None(\Mones)_M$, and let $N \geq 1$ such that $\opH^i(\Mone,\phi^*\Lambda) = 0$ for all $i \geq N$. Let $\fm = \subgrp{y}$ be the maximal homogeneous ideal in $H(\Mones,k)$ that defines the line in $\abs{\Mones}$ through $\wh{\phi}$. To prove that $\wh{\phi} \notin \abs{\Mones}_M$, we will prove the equivalent condition that $\Hbul(\Mones,\Lambda)_{\fm} = 0$.

The homomorphism $\phi$ admits the factorization $\phi = \iota \circ q: \Mone \twoheadrightarrow \Gaone \hookrightarrow \Mones$, where $\iota$ is the homomorphism discussed in Remark \ref{rem:Hom}\eqref{item:iota}. By the discussion of Section \ref{subsubsec:LHSq}, the map in cohomology induced by $q$, $\Hbul(\Gaone,\iota^*\Lambda) \rightarrow \Hbul(\Mone,\phi^*\Lambda)$, is an injection. Then $\opH^j(\Gaone,\iota^*\Lambda) = 0$ for $j \geq N$. In particular, $\Hbul(\Gaone,\iota^*\Lambda)$ is finite-dimensional. Now consider the LHS spectral sequence for the extension $\Gaone \stackrel{\iota}{\rightarrow} \Mones \stackrel{\pi}{\twoheadrightarrow} \M_{1;s-1}$, where now $\pi: \Mones \twoheadrightarrow \M_{1;s-1}$ denotes the quotient homomorphism whose comorphism is the subalgebra inclusion $k[\M_{1;s-1}] \hookrightarrow k[\Mones]$:
	\begin{equation} \label{eq:LHSiota}
	E_2^{i,j} = \opH^i(\M_{1;s-1},\opH^j(\Gaone,\iota^*\Lambda)) \Rightarrow \opH^{i+j}(\Mones,\Lambda).
	\end{equation}
Since $\Hbul(\Gaone,\iota^*\Lambda)$ is finite-dimensional, one gets by \cite[Theorem 1.1]{Drupieski:2016} that the $E_2$-page is finite as a module over the cohomology ring $\Hbul(\Mones,k)$, and hence gets by \cite[Lemma 1.6]{Friedlander:1997} and \cite[III.2.9 Corollary 1]{Bourbaki:1998} that the inflation map $\Hbul(\M_{1;s-1},k) \rightarrow \Hbul(\Mones,k)$ makes $\Hbul(\Mones,\Lambda)$ into a finite module over $\Hbul(\M_{1;s-1},k)$. By \cite[Lemma 3.2.4(3)]{Drupieski:2017a}, the inflation map $\Hbul(\M_{1;s-1},k) \rightarrow \Hbul(\Mones,k)$ satisfies $y \mapsto y$ and $w_{s-1} \mapsto 0$, so $\Hbul(\Mones,\Lambda)$ is finite as a module over the subalgebra $k[y]$ of $H(\Mones,k)$. Fix a homogeneous generating set for $\Hbul(\Mones,\Lambda)$ as a $k[y]$-module, and let $m$ be the maximum of the cohomological degrees of the generators. Then $1_M. (w_s)^m \in \opH^{2m}(\Mones,\Lambda)$. Writing $1_M.(w_s)^m$ as a $k[y]$-linear combination of the homogeneous generators, it follows by degree consideration that each term in the sum is a multiple of $y$, and hence $1_M.(w_s)^m = z.y$ for some $z \in \opH^{2m-1}(\Mones,\Lambda)$. Now localizing at the maximal homogeneous ideal $\fm = \subgrp{y}$, and using the fact that $w_s$ is a unit in $\Hbul(\Mones,k)_{\fm}$, it follows as in the proof of (1) that $\Hbul(\Mones,\Lambda)_{\fm} = 0$.

\subsubsection{Proof of Theorem \ref{theorem:supportheightone} for $G = \Moneone = \Gaone \times \Gam$}

The proof in this case is similar to that for $G = \Mones$ when $s \geq 2$, except that now there are three orbits of $\Aut(\Moneone)$ in $\None(\Moneone)$ to consider, namely, the orbits of the canonical quotient map $\phi_{(1,1)}: \Mone \twoheadrightarrow \Moneone$, and the orbits of the homomorphisms $\phi_{(0,1)}$ and $\phi_{(1,0)}$ discussed in Remark \ref{rem:Hom}\eqref{item:1001factorizations}. Then it suffices to show that
	\begin{enumerate}
	\item $\wh{\phi}_{(1,1)} \in \abs{\Moneone}_M$ if and only if $\phi_{(1,1)} \in \None(\Moneone)_M$,
	\item $\wh{\phi}_{(0,1)} \in \abs{\Moneone}_M$ if and only if $\phi_{(0,1)} \in \None(\Moneone)_M$, and
	\item $\wh{\phi}_{(1,0)} \in \abs{\Moneone}_M$ if and only if $\phi_{(1,0)} \in \None(\Moneone)_M$.
	\end{enumerate}
The proofs of (1) and (2) are entirely similar to the arguments in Section \ref{subsubsec:proofMones} above, with \eqref{eq:LHSiota} replaced by the LHS spectral sequence for the extension $\Gaone \hookrightarrow \Moneone \twoheadrightarrow \Gam$. The proof of (3) is also similar to that of (2) in Section \ref{subsubsec:proofMones}, except that one now applies the observation at the end of Section \ref{subsubsec:LHSq-} when considering the map in cohomology induced by the quotient $\Mone \twoheadrightarrow \Gam$, the spectral sequence \eqref{eq:LHSiota} is replaced by the LHS spectral sequence for the extension $\Gam \hookrightarrow \Moneone \twoheadrightarrow \Gaone$, and the maximal homogeneous ideal $\fm = \subgrp{y} \subset H(\Mones,k)$ is replaced by the maximal homogeneous ideal $\subgrp{x_1} \subset H(\Moneone,k)$.

\subsubsection{Proof of Theorem \ref{theorem:supportheightone} for $G = \Gaone$}

In this case $H(\Gaone,k) \cong k[x_1]$, and hence $\abs{\Gaone}$ and $\None(\Gaone)$ are both one-dimensional affine spaces, and the following are equivalent:
	\begin{itemize}
	\item $\abs{\Gaone}_M = \abs{\Gaone}$
	\item $\opH^i(\Gaone,\Lambda) = \Ext_{\Gaone}^i(M,M) \neq 0$ for infinitely many $i \geq 0$.
	\end{itemize}
Since $\Hbul(\Gaone,\Lambda)$ embeds into $\Hbul(\Mone,\Lambda)$ by the observation at the end of Section \ref{subsubsec:LHSq}, it follows that if $\None(\Gaone)_M = \set{0}$, then $\opH^i(\Gaone,\Lambda) = 0$ for $i \gg 0$, and hence $\abs{\Gaone}_M = \set{0}$. Conversely, suppose $\None(\Gaone)_M = \None(\Gaone)$, so that $\opH^i(\Mone,q^*\Lambda) \neq 0$ for infinitely many $i \geq 0$. Consider the commutative diagram
\[
\xymatrix@C+.5em{
H(\Gaone,k) \ar@{->}[r]^{\rho_{\Lambda}} \ar@{->}[d]^{q^*} & \Hbul(\Gaone,\Lambda) \ar@{->}[d]^{q^*} \\
H(\Mone,k) \ar@{->}[r]^{\rho_{q^* \Lambda}} & \Hbul(\Mone,q^* \Lambda).
}
\]
As in the proof of (1) in Section \ref{subsubsec:proofMones} above, it follows that $\rho_{q^*\Lambda}(y^n) \neq 0$ for $n \geq 1$. Since $q^*(x_1) = x_1$ by \cite[Lemma 3.2.4]{Drupieski:2017a}, and since $x_1 = y^2$ in $H(\Mone,k)$, this implies that $\ker(\rho_\Lambda) = \set{0}$, and hence implies that $\abs{\Gaone}_M = \abs{\Gaone}$.

\subsubsection{Proof of Theorem \ref{theorem:supportheightone} for $G = \Gam$} The reasoning is entirely similar to that for $G = \Gaone$.

\subsection{Applications of the BIKP detection theorem} \label{subsection:BIKPapplications}

The following theorem of Benson, Iyengar, Krause, and Pevtsova (BIKP) was announced by Julia Pevtsova at the Mathematical Congress of the Americas in Montreal in July 2017.

\begin{theorem}[BIKP Detection Theorem] \label{theorem:BIKP}
Let $G$ be a finite unipotent $k$-super\-group scheme. Given $z \in \opH^n(G,k)$ and a field extension $K/k$, set $z_K = z \otimes_k 1 \in \opH^n(G,k) \otimes_k K = \opH^n(G \otimes_k K,K)$. Given a rational $G$-supermodule $M$, set $M_K = M \otimes_k K$. Then the following hold:
	\begin{enumerate}
	\item \label{item:detectnilpotents} A cohomology class $z \in \Hbul(G,k)$ is nilpotent if and only if for every field extension $K$ of $k$ and every elementary subsupergroup scheme $E$ of $G \otimes_k K$, the restriction of $z_K \in \Hbul(G \otimes_k K,K)$ to $\Hbul(E,K)$ is nilpotent.
	\item \label{item:detectprojectivity} Let $M$ be a rational $G$-supermodule. Then $M$ is projective if and only if for every field extension $K$ of $k$ and every elementary subsupergroup scheme $E$ of $G \otimes_k K$, the restriction of $M_K$ to $E$ is projective.
	\end{enumerate}
\end{theorem}

In the remainder of the paper we will discuss a few of the ways we expect the BIKP detection theorem may be applied (\`{a} la Suslin, Friedlander, and Bendel \cite{Suslin:1997,Suslin:1997a}) to generalize our results in Corollary \ref{corollary:psirelementary} and Theorem \ref{theorem:supportheightone}. Since at the time of writing this article the BIKP detection theorem has only been announced as a preliminary result, we will state these expected applications only as conjectures.

\begin{conjecture} \label{conjecture:spectrum}
Suppose $k$ is algebraically closed, and let $G$ be an infinitesimal unipotent $k$-supergroup scheme of height $ \leq r$. Then the kernel of the algebra homomorphism
	\[
	\psi_r: H(G,k) \rightarrow k[\NrG]
	\]
is nilpotent, and its image contains the $p^r$-th power of each element of $k[\NrG]$. Consequently, the associated morphism of varieties
	\[
	\Psi = \Psi_r: \NrG \rightarrow \abs{G} = \Max(H(G,k))
	\]
is a homeomorphism.
\end{conjecture}

We can already show (for arbitrary infinitesimal $G$) that $\psi_r$ is surjective onto $p^r$-th powers.

\begin{lemma} \label{lemma:psirontopr}
Suppose $k$ is algebraically closed, and let $G$ be an infinitesimal $k$-supergroup scheme of height $ \leq r$. Then the image of the homomorphism
	\[
	\psi_r: H(G,k) \rightarrow k[\NrG]
	\]
contains the $p^r$-th power of each element of $k[\NrG]$. 
\end{lemma}

\begin{proof}
Let $\iota: G \hookrightarrow \GLmn$ be a closed embedding, and set $N = \max(m,n)$. Then by Lemma \ref{lem:HomMrGvariety}, $\NrG = \bfHom(\M_{r;N},G)(k)$. The embedding of $G$ into $\GLmn$ defines a closed embedding of affine $k$-superschemes $\bfHom(\M_{r;N},G) \subset \bfHom(\M_{r;N},\GLmn)$, and hence a surjection of coordinate superalgebras $k[\bfHom(\M_{r;N},\GLmn)] \twoheadrightarrow k[\bfHom(\M_{r;N},G)]$. Passing to the largest purely even quotients, the embedding induces surjective $k$-algebra homomorphism $k[V_{r;N}(\GLmn)] \twoheadrightarrow k[V_{r;N}(G)]$. Next, by \cite[Lemma 4.4.1]{Drupieski:2013b}, the image of $\iota: G \hookrightarrow \GLmn$ is contained in $\GLmnr$, the $r$-th Frobenius kernel of $\GLmn$. Then by \cite[Lemma 6.2.1]{Drupieski:2017a}, there exists a commutative diagram
\[
\xymatrix{
H(\GLmnr,k) \ar@{->}[r]^{\psi_{r;N}} \ar@{->}[d]^{\iota^*} & k[V_{r;N}(\GLmn)] \ar@{->>}[d]^{\iota^*} \\
H(G,k) \ar@{->}[r]^{\psi_{r;N}} & k[V_{r;N}(G)].
}
\]
The top horizontal map is surjective onto $p^r$-th powers by \cite[Theorem 6.2.3]{Drupieski:2017a}, so the commutativity of the diagram implies that the bottom horizontal map is surjective onto $p^r$-th powers as well. Finally, the commutative diagram \eqref{eq:psirdiagram} implies that $\psi_r$ is surjective onto $p^r$-th powers.
\end{proof}

To argue that $\ker(\psi_r)$ is nilpotent, we first need a strengthening of Proposition \ref{prop:identifyhom} and of Corollary \ref{corollary:psirelementary}. Recall from \cite[Remark 3.1.3(5)]{Drupieski:2017a} that the super Frobenius morphism $\bsF^\ell: \M_{r+\ell} \rightarrow \Mr$ is the supergroup homomorphism whose comorphism $(\bsF^\ell)^*: k[\Mr] \rightarrow k[\M_{r+\ell}]$ is defined by $\theta \mapsto \theta^{p^\ell}$, $\tau \mapsto \tau$, and $\sigma_i \mapsto \sigma_i$ for $i \in \N$. Note that $\theta^{p^{r-1}} = \sigma_1$ in $k[\Mr]$, but $\theta^{p^{r+\ell-1}} = \sigma_1$ in $k[\M_{r+\ell}]$, so the assignment $(\bsF^\ell)^*(\theta) = \theta^{p^\ell}$ does make sense. We could have incorporated the first statement of the following lemma into Proposition \ref{prop:identifyhom}, but to do so would have made the proof more notationally cumbersome than it already is.

\begin{lemma} \label{lemma:Homr+ell}
Let $G$ be a height-$r$ infinitesimal elementary $k$-supergroup scheme, and let $A = \Azero$ be a reduced purely even commutative $k$-algebra. Then for all $\ell \geq 1$, composition with the super Frobenius morphism $\bsF^\ell: \M_{r+\ell} \rightarrow \Mr$ defines a bijection
	\[
	\bfHom(\Mr,G)(A) \stackrel{\sim}{\rightarrow} \bfHom(\M_{r+\ell},G)(A).
	\]
In particular, if $k$ is algebraically closed then $\NrG \cong \calN_{r+\ell}(G)$.
\end{lemma}

\begin{proof}
The coordinate algebra $A[\Mr]$ identifies via the comorphism $(\bsF^\ell)^*: A[\Mr] \rightarrow A[\M_{r+\ell}]$ with a Hopf subsuperalgebra of $A[\M_{r+\ell}]$. Then the proof proceeds by the same reasoning as the proof of Proposition \ref{prop:identifyhom}: One begins by classifying all Hopf superalgebra homomorphisms $\phi: A[G] \rightarrow A[\M_{r+\ell}]$, and then one observes from the classification that the image of each homomorphism is contained in the subalgebra $A[\Mr]$ of $A[\M_{r+\ell}]$, and hence each homomorphism $\phi: A[G] \rightarrow A[\M_{r+\ell}]$ uniquely factors through the comorphism $(\bsF^\ell)^*: A[\Mr] \rightarrow A[\M_{r+\ell}]$. The classification of all Hopf superalgebra homomorphisms $\phi: A[G] \rightarrow A[\M_{r+\ell}]$ follows from essentially a repetition of the argument in the proof Proposition \ref{prop:identifyhom}. The first difference, which is really just an issue of bookkeeping, occurs when considering the $\Z$-degree of elements in $A[\M_{r+\ell}]$, as when analyzing the formulas \eqref{eq:coproductsigmap}, \eqref{eq:pn1coproduct}, \eqref{eq:coproducttheta}, and \eqref{eq:coproductthetas=1}. The $\Z$-grading on $k[\M_{r+\ell}]$ is defined instead by $\deg(\theta) = 2$, $\deg(\sigma_j) = 2j p^{r+\ell-1}$, and $\deg(\tau) = p^{r+\ell}$ (so $(\bsF^\ell)^*$ is a morphism of graded algebras that multiplies degrees by $p^\ell$), but this difference in $\Z$-gradings between $k[\Mr]$ and $k[\M_{r+\ell}]$ does not alter the substance of any of the arguments in the proof of Proposition \ref{prop:identifyhom}. The second difference comes in considering the space of primitive elements in $A[\M_{r+\ell}]$. An element $z \in A[\M_{r+\ell}]$ is primitive if and only if it is an $A$-linear combination of $\theta,\theta^p,\ldots,\theta^{p^{r+\ell-1}}$. Then the only $p$-nilpotent primitive elements in $A[\M_{r+\ell}]$ are multiples of $\theta^{p^{r+\ell-1}} = \sigma_1$, and the only $p^r$-nilpotent primitive elements in $A[\M_{r+\ell}]$ are $A$-linear combinations of $\theta^{p^\ell},\theta^{p^{\ell+1}},\ldots,\theta^{p^{r+\ell-1}}$. Thus when considering the primitive components of $\phi(\sigma_p)$, $\phi(\sigma_{p^{n+1}})$, or $\phi(\theta)$, as in \eqref{eq:coproductsigmap}, \eqref{eq:pn1coproduct}, \eqref{eq:coproducttheta}, and \eqref{eq:coproductthetas=1}, one only gets primitive elements in the image of $(\bsF^\ell)^*$.
\end{proof}

\begin{corollary} \label{corollary:psirheightleqr}
Suppose $k$ is algebraically closed, and let $G$ be an infinitesimal elementary $k$-super\-group scheme of height $\leq r$. Then the algebra homomorphism
	\[
	\psi_r: H(G,k) \rightarrow k[\NrG]
	\]
is injective modulo nilpotents and its image contains all $p^r$-th powers in $k[\calN_r(G)]$.
\end{corollary}

\begin{proof}
Suppose $G$ is infinitesimal of height $r' \leq r$. By Lemma \ref{lemma:Homr+ell}, the super Frobenius morphism defines an identification of coordinate algebras $k[\NrG] = k[\calN_{r'}(G)]$. Lemma \ref{lemma:Homr+ell} also implies that the universal supergroup homomorphism from $\Mr$ to $G$ factors through the super Frobenius morphism $\Mr \otimes_k k[\NrG] \rightarrow \M_{r'} \otimes_k k[\NrG]$. Then there exists a commutative diagram
\[
\xymatrix@C=1em{
H(G,k) \ar@{->}[r] \ar@{=}[d] & \Hbul(G \otimes_k k[\calN_{r'}(G)],k[\calN_{r'}(G)]) \ar@{->}[r] \ar@{=}[d] & \Hbul(\M_{r'} \otimes_k k[\calN_{r'}(G)],k[\calN_{r'}(G)]) \ar@{->}[r] \ar@{->}[d] & k[\calN_{r'}(G)] \ar@{=}[d] \\
H(G,k) \ar@{->}[r] & \Hbul(G \otimes_k k[\NrG],k[\NrG]) \ar@{->}[r] & \Hbul(\Mr \otimes_k k[\NrG],k[\NrG]) \ar@{->}[r] & k[\NrG]
}
\]
in which the rows are the composites defining $\psi_{r'}$ and $\psi_r$, respectively, and the third vertical arrow is the map in cohomology induced by the super Frobenius morphism. The top row of the diagram is injective modulo nilpotents and surjective onto $p^r$-th powers, by Corollary \ref{corollary:psirelementary}, so the commutativity of the diagram implies the corresponding conclusion for the bottom row.
\end{proof}

Now we explain how the previous results, when combined with the BIKP detection theorem, should imply the validity of Conjecture \ref{conjecture:spectrum}.

\begin{proof}[Justification for Conjecture \ref{conjecture:spectrum}]
Suppose $k$ is algebraically closed, and let $G$ be an infinitesimal unipotent $k$-supergroup scheme of height $\leq r$. We have already shown in Lemma \ref{lemma:psirontopr} that the image of $\psi_r$ contains the $p^r$-th power of each element of $k[\NrG]$, so let $z \in H(G,k)$ be a homogeneous element, and suppose $\psi_r(z) = 0$. Then for each (algebraically closed) field extension $K/k$ and each (infinitesimal) elementary subsupergroup scheme $E$ of $G \otimes_k K$, there exists by Proposition \ref{prop:psirgradedmap} a commutative diagram
\[
\xymatrix@C+1.5em{
H(G \otimes_k K,K) \ar@{->}[r]^{\psi_r \otimes_k K} \ar@{->}[d]^{\iota^*} & K[\NrG] \ar@{->}[d]^{\iota^*} \\
H(E,K) \ar@{->}[r]^{\psi_r} & K[\Nr(E)].
}
\]
Here $\iota: E \hookrightarrow G \otimes_k K$ denotes the closed embedding of $E$ into $G \otimes_k K$. Since $\psi_r(z) = 0$, the commutativity of the diagram implies that $(\psi_r \circ \iota^*)(z_K) = 0$. But the bottom arrow of the diagram is injective modulo nilpotents by Corollary \ref{corollary:psirheightleqr}, so $\iota^*(z_K)$ must be nilpotent in $H(E,K)$. Since $K$ and $E$ were arbitrary, this implies by the BIKP detection theorem that $z$ is nilpotent.
\end{proof}

\begin{conjecture} \label{conjecture:supportheightone}
Suppose $k$ is algebraically closed. Let $G$ be a height-$1$ infinitesimal unipotent $k$-super\-group scheme, and let $M$ be a finite-dimensional rational $G$-supermodule. Then under the homeomorphism $\Psi = \Psi_1: \NoneG \rightarrow \abs{G}$ of Conjecture \ref{conjecture:spectrum},
	\[
	\Psi^{-1}(\abs{G}_M) = \NoneG_M = \set{ \phi \in \NoneG : \id_{\Mone}(\phi^*M) = \infty}.
	\]
\end{conjecture}

To justify this conjecture we must assume a stronger version of Theorem \ref{theorem:BIKP}\eqref{item:detectnilpotents}:
	\begin{itemize}
	\item[(BIKP $1'$)] Let $\Lambda$ be a unital associative rational $G$-algebra, and let $z \in \Hbul(G,\Lambda)$. Then $z$ is nilpotent if and only if for every field extension $K$ of $k$, and every elementary subsupergroup scheme $E$ of $G \otimes_k K$, the restriction of $z_K \in \Hbul(G \otimes_k K,\Lambda_K)$ to $\Hbul(E,\Lambda_K)$ is nilpotent.
	\end{itemize}
We also need to assume an affirmative answer to Question \ref{question:supportsetclosed}, i.e., we must assume that $\NoneG_M$ is a Zariski closed subset of $\NoneG$, and we need to assume that the ideal of functions $\calI_M = \calI(G)_M \subset k[\NoneG]$ vanishing on $\NoneG_M$ is compatible with field extensions. Specifically, let $K$ be an algebraically closed field extension of $k$. Then $M_K := M \otimes_k K$ is a rational $G \otimes_k K$-super\-module, and one can consider the set
	\[
	\None(G \otimes_k K)_{M_K} = \set{ \phi \in \bfHom(\Mone \otimes_k K,G \otimes_k K)(K): \id_{\Mone \otimes_k K}(\phi^*M_K) = \infty}.
	\]
Then we want to assume that the image of $\calI(G)_M$ in $K[\None(G \otimes_k K)]$ under base change is contained in the ideal of functions $\calI(G \otimes_k K)_{M_K} \subset K[\None(G \otimes_k K)]$ vanishing on $\None(G \otimes_k K)_{M_K}$.

\begin{proof}[Justification (modulo assumptions) for Conjecture \ref{conjecture:supportheightone}]
Given the preceding assumptions, set $\Lambda = \Hom_k(M,M)$, and let $I_M = \ker(\rho_\Lambda)$ be the defining ideal of $\abs{G}_M$. By the assumption that $\NoneG_M$ is closed in $\NoneG$, the ideal $\calI_M$ is a radical ideal and is the kernel of the canonical quotient map $k[\NoneG] \twoheadrightarrow k[\NoneG_M]$. Now to prove Conjecture \ref{conjecture:supportheightone}, it suffices as in the proof of \cite[Corollary 6.3.1]{Suslin:1997a} to show that the homomorphism $\psi = \psi_1 : H(G,k) \rightarrow k[\NoneG]$ satisfies
	\begin{equation} \label{eq:psiinverse}
	\psi^{-1}(\calI_M) = \sqrt{I_M}.
	\end{equation}
It follows from Theorem \ref{theorem:supportheightone} that \eqref{eq:psiinverse} is true in the special case that $G$ is a height-one infinitesimal elementary supergroup scheme.

Let $\phi \in \NoneG = \bfHom(\Mone,G)(k)$, and let $z \in H(G,k)$ be homogeneous of degree $n$. Then by the definition of $\psi$, $\phi^*(z) = \psi(z)(\phi) \cdot y^n \in H(\Mone,k)$; cf.\ \eqref{eq:uG*r=1}. Now suppose $\phi \in \NoneG_M$ but that $\psi(z)(\phi) \neq 0$, so that $\psi(z) \notin \calI_M$. Then $\phi^*(z)$ is a nonzero multiple of $y^n$ in $H(\Mone,k)$. Since $\phi \in \NoneG_M$, then $\opH^i(\Mone,\phi^*\Lambda) \neq 0$ for infinitely many $i \geq 1$. But $\Hbul(\Mone,\phi^*\Lambda)$ is finite over $H(\Mone,k) \cong k[y]$ by Lemma \ref{lemma:Mrfg}, so this implies that $\rho_{\phi^* \Lambda}(y) \in H(\Mone,\phi^* \Lambda)$ is not nilpotent, and hence $\rho_{\phi^*\Lambda}(\phi^*(z))$ is not nilpotent. Since $\rho_{\phi^*\Lambda}(\phi^*(z)) = \phi^*(\rho_\Lambda(z))$, we deduce that $\rho_\Lambda(z)$ is not nilpotent, and hence $z \notin \sqrt{I_M}$. Thus, if $\psi(z) \notin \calI_M$, then $z \notin \sqrt{I_M}$.

For the converse, suppose $\psi(z) \in \calI_M$. Let $K/k$ be an extension of (algebraically closed) fields, and let $E$ be an elementary (height-one infinitesimal) subsupergroup scheme of $G \otimes_k K$. Write $\nu: E \hookrightarrow G \otimes_k K$ for the inclusion of $E$ into $G \otimes_k K$. Then we want to show that the restricted cohomology class $\nu^*(\rho_{\Lambda_K}(z_K)) = \rho_{\nu^* \Lambda_K}(\nu^* (z_K)) \in \Hbul(E,\nu^* \Lambda_K)$ is nilpotent. Write $\psi(z)_K$ for the image of $\psi(z)$ under the base change map $k[\NoneG] \rightarrow K[\None(G \otimes_k K)]$ of Remark \ref{remark:kNrG}\eqref{item:basechange}. By the assumption that the ideal $\calI_M = \calI(G)_M$ is compatible with field extensions, we get $\psi(z)_K \in \calI(G \otimes_k K)_{M_K}$. Next, it is evident from the definitions that composition with $\nu$ maps $\None(E)_{\nu^* M_K}$ into the support set $\None(G \otimes_k K)_{M_K}$, and hence the restriction map $\nu^*: K[\None(G \otimes_k K)] \rightarrow K[\None(E)]$ maps $\calI(G \otimes_k K)_{M_K}$ into the ideal $\calI(E)_{\nu^* M_K}$. Then by the naturality of $\psi$,
	\[
	\psi(\nu^*(z_K)) = \nu^*(\psi(z_K)) = \nu^*(\psi(z)_K) \in \calI(E)_{\nu^* M_K}.
	\]
Since \eqref{eq:psiinverse} is true for height-one infinitesimal elementary supergroup schemes, this means that the cohomology class $\rho_{\nu^* \Lambda_K}(\nu^*(z_K)) \in \Hbul(E,\nu^* \Lambda_K)$ is nilpotent. Thus, since $K$ and $E$ were arbitrary, we deduce by (BIKP $1'$) that if $\psi(z) \in \calI_M$, then $z \in \sqrt{I_M}$.
\end{proof}



\makeatletter
\renewcommand*{\@biblabel}[1]{\hfill#1.}
\makeatother

\bibliographystyle{eprintamsplain}
\bibliography{infinitesimal-unipotent-supergroups}

\end{document}